\documentclass[a4paper]{amsart}

\usepackage{amssymb}
\usepackage[dvipdfmx]{hyperref}

\usepackage{mathrsfs}

\numberwithin{equation}{section}

\newtheorem{theorem}{Theorem}[section]
\newtheorem{proposition}[theorem]{Proposition}
\newtheorem{lemma}[theorem]{Lemma}

\theoremstyle{definition}
\newtheorem{definition}[theorem]{Definition}
\newtheorem{example}[theorem]{Example}

\theoremstyle{remark}
\newtheorem{remark}[theorem]{Remark}

\newtheoremstyle{assumption}
{}	
{}	
{}
{}
{\bfseries}
{}
{.5em}
{(\thmname{#1}\thmnumber{#2})}%

\theoremstyle{assumption}
\newtheorem{assumption}{A}

\newcommand{\tref}[2][]
	{\ifx#1""{Theorem~\textup{\ref{#2}}}%
	\else{Theorems~\textup{\ref{#1}} and~\textup{\ref{#2}}}\fi}
\newcommand{\trefs}[4][]
	{\ifx#1""{Theorems~\textup{\ref{#2}},~\textup{\ref{#3}} and~\textup{\ref{#4}}}%
	\else{Theorems~\textup{\ref{#1}},~\textup{\ref{#2}},~\textup{\ref{#3}} and~\textup{\ref{#4}}}\fi}
\newcommand{\pref}[2][]
	{\ifx#1""{Proposition~\textup{\ref{#2}}}%
	\else{Propositions~\textup{\ref{#1}} and~\textup{\ref{#2}}}\fi}
\newcommand{\prefs}[4][]
	{\ifx#1""{Propositions~\textup{\ref{#2}},~\textup{\ref{#3}} and~\textup{\ref{#4}}}%
	\else{Propositions~\textup{\ref{#1}},~\textup{\ref{#2}},~\textup{\ref{#3}} and~\textup{\ref{#4}}}\fi}
\newcommand{\lref}[2][]
	{\ifx#1""{Lemma~\textup{\ref{#2}}}%
	\else{Lemmas~\textup{\ref{#1}} and~\textup{\ref{#2}}}\fi}
\newcommand{\lrefs}[4][]
	{\ifx#1""{Lemma~\textup{\ref{#2}}, \textup{\ref{#3}} and~\textup{\ref{#4}}}%
	\else{Lemmas~\textup{\ref{#1}}, \textup{\ref{#2}}, \textup{\ref{#3}} and \textup{\ref{#4}}}\fi}
\newcommand{\dref}[2][]
	{\ifx#1""{Definition~\textup{\ref{#2}}}%
	\else{Definition~\textup{\ref{#1}} and~\textup{\ref{#2}}}\fi}
\newcommand{\rref}[2][]
	{\ifx#1""{Remark~\textup{\ref{#2}}}%
	\else{Remark~\textup{\ref{#1}} and~\textup{\ref{#2}}}\fi}
\newcommand{\iref}[2][]
	{\ifx#1""{\textup{(\ref{#2})}}%
	\else{\textup{(\ref{#1})}~and~\textup{(\ref{#2})}}\fi}
\newcommand{\secref}[2][]
	{\ifx#1""{Section~\textup{\ref{#2}}}%
	\else{Sections~\textup{\ref{#1}} and~\textup{\ref{#2}}}\fi}

\newcommand{\aref}[1]{\textup{(A\ref{#1})}}
\newcommand{\exref}[1]{Example~\textup{\ref{#1}}}

\newcommand{\RealNum}{\mathbf{R}}
\newcommand{\NaturalNum}{\mathbf{N}}

\newcommand{\idMat}[1][]{I_{#1}}
\newcommand{\probSp}{\Omega}

\newcommand{\prob}{\boldsymbol{P}}
\newcommand{\expect}{\boldsymbol{E}}
\newcommand{\CM}{\mathfrak{H}}
\newcommand{\Hurst}{H}

\newcommand{\SobSp}[2]{\ifx#2""{\mathbf{D}_{#1}}\else{\mathbf{D}_{{#1},{#2}}}\fi}
\newcommand{\tilSobSp}[2]{\ifx#2""{\tilde{\mathbf{D}}_{#1}}\else{\tilde{\mathbf{D}}_{{#1},{#2}}}\fi}

\newcommand{\var}[1]{{#1}\mathchar`-{\text{\textrm{var}}}}
\newcommand{\Hol}[1]{{#1}\mathchar`-{\text{\textrm{H{\"o}l}}}}
\newcommand{\Bes}[1]{{#1}\mathchar`-{\text{\textrm{Bes}}}}
\newcommand{\wienerChaos}[1]{\mathcal{C}_{#1}}

\newcommand{\geoRPsBes}[1]{G\Omega^{\text{\textrm{B}}}_{#1}}
\newcommand{\geoRPsHol}[1]{G\Omega^{\text{\textrm{H}}}_{#1}}
\newcommand{\geoRPs}[1]{G\Omega_{#1}}
\newcommand{\RP}[1]{\boldsymbol{#1}}

\newcommand{\indicator}[1]{\mathtt{1}_{#1}}
\DeclareMathOperator{\MinEigenVal}{\lambda_{\min}}
\DeclareMathOperator{\MaxEigenVal}{\lambda_{\max}}
\newcommand{\intPart}[1]{\lfloor #1 \rfloor}

\newcommand{\vecFields}{\mathcal{V}}

\begin{document}

\title[Asymptotic expansion of density]
{
Asymptotic expansion of the density
for hypoelliptic rough differential equation
}
%
\author[Y. Inahama]{Yuzuru Inahama}
\address[Y. I.]{Graduate School of Mathematics, Kyushu University. Motooka, Nishi-ku, Fukuoka, 819-0395, Japan}
\email{inahama@math.kyushu-u.ac.jp}

\author[N. Naganuma]{Nobuaki Naganuma}
\address[N. N.]{Graduate School of Engineering Science, Osaka University. Toyonaka, Osaka, 560-8531, Japan}
\email{naganuma@sigmath.es.osaka-u.ac.jp}




\maketitle


\begin{abstract}
We study a rough differential equation
driven by fractional Brownian motion with Hurst parameter
$\Hurst$ $(1/4<\Hurst \le 1/2)$.
Under H\"ormander's condition on the coefficient
vector fields,
the solution has a smooth density for each fixed time.
Using Watanabe's distributional Malliavin calculus,
we obtain a short time full asymptotic expansion of the density
under quite natural assumptions.
Our main result can be regarded as
a ``fractional version" of Ben Arous' famous work
on the off-diagonal asymptotics.
\end{abstract}

\section{Introduction}
In this paper we study  from the viewpoint of Malliavin calculus
the following rough differential
equation (RDE) driven by $d$-dimensional
 fractional Brownian motion (fBm) $(w_t)$
with Hurst parameter $\Hurst \in (1/4, 1/2]$.
\begin{align*}
	dy_t
	=
		\sum_{i=1}^d
			V_i (y_t) dw^i_t
	 	+
		V_0 (y_t)\,
		dt
	\quad
	\text{with}
	\quad
	y_0=a \in \RealNum^n.
\end{align*}
Here, $V_i$ $(0 \leq i \leq d)$ is a
sufficiently regular vector fields on $\RealNum^n$.
When $\Hurst=1/2$,
fBm is the usual Brownian motion and
this RDE coincides with the usual
stochastic differential equation (SDE) of Stratonovich type.

Malliavin calculus for RDEs driven by
fractional Brownian rough path or
 Gaussian rough path is a quite active topic now and
a number of papers were published on it recently.
See \cite{
BaudoinNualartOuyangTindel016,
BaudoinOuyangZhang2015,
BaudoinOuyangZhang2016,
BoedihardjoGengQian2016,
CassFriz2010,
CassFrizVictoir2009,
CassHairerLittererTindel2015,
Driscoll2013,
HairerPillai2013,
HuTindel2013,
Inahama2014,
Inahama2016c}
among others.

Due to \cite{CassHairerLittererTindel2015, Inahama2014}
and the general theory of Malliavin calculus,
under H\"ormander's bracket generating condition on
$V_i$ $(0 \leq i \leq d)$ at the starting point $a$,
the solution $y_t$ has a smooth density $p_t (a, a')$
with respect to the Lebesgue measure
for every $t>0$, that is, the function $a' \mapsto p_t (a, a')$ is smooth
and $\prob(y_t \in A)=\int_A p_t (a, a')\,da'$ holds for every Borel subset $A \subset \RealNum^n$.

In this paper we are interested in short time asymptotics of
this density function.
We will prove in \tref{thm_MAIN}
 a full asymptotic expansion of
$p_t (a, a')$ as $t \searrow 0$ for $a \neq a'$
under, loosely speaking,
 H\"ormander's  condition at $a$ and
 the ``unique minimizer" condition
 (see \aref{ass_HorCon} and \aref{ass_minimizer}
 below, respectively).

This kind of short time asymptotic expansion of the density
(under the unique minimizer condition)  was
first shown for $\Hurst \in (1/2, 1)$
in the framework of Young integration theory
by \cite{BaudoinOuyang2011, Inahama2016b},
then for $\Hurst \in (1/3, 1/2]$ by \cite{Inahama2016c} in the
framework of  rough path theory.
These results are not completely satisfactory, however,
for the following two reasons.
First, from the viewpoint of rough path theory,
the condition on
Hurst parameter should be $\Hurst>1/4$.
Second, in these papers the ellipticity assumption
on the vector fields is assumed.
From the viewpoint of Malliavin calculus,
it should be replaced by H\"ormander's condition.

The purpose of the present paper is
to generalize this kind of off-diagonal asymptotic expansion
to a satisfactory form by refining the arguments in \cite{Inahama2016c}.
Hence, this is a continuation of \cite{Inahama2016c}
and our proof, just like the one in \cite{Inahama2016c},
 is based on Watanabe's distributional Malliavin calculus
\cite{Watanabe1987, IkedaWatanabe1989}.

Due to the generalizations, however,
many parts of our proof become more complicated than
their counterparts in \cite{Inahama2016c}.
Examples are as follows.
In \secref{sec_1565918229} we calculate the Young translation
of Besov rough path for the third level case.
Since we work under H\"ormander's condition,
we must calculate Malliavin covariance matrices
more carefully in \secref{sec_1545285261}.
In particular, we prove Kusuoka-Stroock type estimate (\pref{prop_1526540975})
and the uniform non-degeneracy for  the scaled-shifted RDE (\pref{prop_1526956280}).

When $\Hurst=1/2$, $p_t (a, a')$ can also be viewed as
the heat kernel
of the corresponding parabolic equation
and has been extensively studied
by the analytic and the probabilistic methods.
When it comes to off-diagonal asymptotic expansion
under the unique minimizer condition,
Ben Arous \cite{BenArous1988b} seems to be the most famous.
In the special case $\Hurst=1/2$,
our main result (\tref{thm_MAIN})
recovers the main result in \cite{BenArous1988b}.
Therefore, \tref{thm_MAIN} can be regarded as
an ``fBm-version" of \cite{BenArous1988b}.

The organization of this paper is as follows.
In \secref{sec_1545285157} we introduce the setting and
assumptions and then state our main result.
In \secref{sec_1545285214} we gather basic results
in rough path theory for later use, including
many probabilistic properties of
fractional Brownian rough path.
In \secref{sec_1545285241},  a Taylor-like expansion of Lyons-It\^o map,
both in the deterministic and probabilistic senses, is given.
In \secref{sec_1545285261} we apply Malliavin calculus
to an RDE driven by fractional Brownian rough path
with Hurst parameter $\Hurst \in (1/4, 1/2]$.
Following \cite{BaudoinOuyangZhang2015}, \cite{CassHairerLittererTindel2015} and \cite{GessOuyangTindel2017},
we carefully prove a few important propositions.
Those are used in \secref{sec.pf}
to prove our main theorem.
In \secref{sec_1545285297} we apply our main result to concrete examples.
\secref{sec_1545285314} is devoted to showing a technical lemma.

In what follows, we use the following notation.
For $p \ge 0$, $\intPart{p}$ denotes the integer part of $p$.
For $1<p<\infty$, $0<\alpha<1$ and a metric space $E$,
$C^{\var{p}}([0,1];E)$ and $C^{\Hol{\alpha}}([0,1];E)$
stand for the space of all continuous paths of bounded $p$-variation
and the space of all $\alpha$-H\"{o}lder paths from $[0,1]$ to $E$, respectively.
We denote by $C_o^{\var{p}}([0,1];E)$ and $C_o^{\Hol{\alpha}}([0,1];E)$
the subset of $C^{\var{p}}([0,1];E)$ and $C^{\Hol{\alpha}}([0,1];E)$, respectively,
of the paths starting from $o \in E$.
For a real symmetric matrix $A$, $\MinEigenVal(A)$ and $\MaxEigenVal(A)$ denote
the smallest and the largest eigenvalues of $A$, respectively.
A smooth vector field $F = \sum_{i=1}^n F^i
(\partial/\partial \eta^i)$ on ${\mathbf R}^n$
is often identified with
the corresponding ${\mathbf R}^n$-valued smooth function
$(F^i (\eta_1, \ldots, \eta_n) )_{1 \le i \le n}$ on ${\mathbf R}^n$.
Its Jacobi matrix is denoted by $\nabla F$, that is,
$
	(\nabla F)_{ij}
	=
		\partial F^i/\partial \eta_j
	=
		\partial F^i/\partial \eta_j
$
$(1 \le i, j \le n)$.


\section{Setting and main result}\label{sec_1545285157}

\subsection{Setting}
In this subsection, we introduce a stochastic process that will play a main role in this paper.
From now on we denote by $(w_t)_{t\geq 0}$
the $d$-dimensional fractional Brownian motion (fBm) with Hurst parameter $\Hurst$.
Throughout this paper we assume $1/4<\Hurst\leq 1/2$.
It is a unique $d$-dimensional, mean-zero, continuous Gaussian process with covariance
$\expect[w^i_s w^j_t]=\delta_{ij} R(s,t)$ ($s,t\geq 0$), where
\begin{align*}
	R(s,t)
	=
		\frac{1}{2}
		\{
			|s|^{2\Hurst}
			+|t|^{2\Hurst}
			-|t-s|^{2\Hurst}
		\}.
\end{align*}
Note that, for any $c>0$, $(w_{ct})_{ t \geq 0}$ and $(c^{\Hurst} w_t)_{ t \geq 0}$ have the same law.
This property is called self-similarity or scale invariance.
When $\Hurst=1/2$, it is the usual Brownian motion.
In what follows the time interval will always be $[0,1]$.
It is well-known that $(w_t)_{0 \le t \le 1}$ admits
a canonical rough path lift $\RP{w}=(\RP{w}^1,\dots,\RP{w}^{\intPart{1/\Hurst}})$,
which is called fractional Brownian rough path (\cite{CoutinQian2002}).

Let $V_i : \RealNum^n \to \RealNum^n$ be $C_b^{\infty}$,
that is, $V_i$ is a bounded smooth function with
bounded derivatives of all order ($0 \leq i \leq d$).
We shall identify $V_i$ with its corresponding vector field
and denote the vector field by the same symbol.
We consider the following RDE:
\begin{align}\label{eq_RDEdrivenByFBM}
	dy_t
	=
		\sum_{i=1}^d
			V_i(y_t)\,
			dw^i_t
		+
		V_0(y_t)\,
		dt
	\qquad
	\text{with}
	\qquad
	y_0=a,
\end{align}
where $a \in \RealNum^n$ is a deterministic initial point.
This RDE is driven by the Young pairing
$(\RP{w}, \RP{\lambda})$, where $\lambda_t =t$.
The unique solution is denoted by $\RP{y}=(\RP{y}^1,\dots,\RP{y}^{\intPart{1/\Hurst}})$
and we set $y_t := a + \RP{y}_{0,t}^1$ as usual.
We will sometimes write $y_t =y_t(a) = y_t(a, \RP{w})$ etc.~to make explicit the dependence on $a$ and $\RP{w}$.
We often use a matrix notation for \eqref{eq_RDEdrivenByFBM},
that is, we set $b=V_0$
and $\sigma=[V_1,\dots,V_d]$, which is $n\times d$ matrix-valued,
and write \eqref{eq_RDEdrivenByFBM} as
\begin{align*}
	dy_t
	=
		\sigma(y_t)\,
		dw_t
		+
		b(y_t)\,
		dt
	\qquad
	\text{with}
	\qquad
	y_0
	=
		a.
\end{align*}

\subsection{Assumptions}
In this subsection, we introduce assumptions of the main theorem.
First, we introduce the H{\"o}rmander condition:
\begin{definition}[H{\"o}rmander condition]\label{def_1541388669}
	Set
	\begin{align*}
			\vecFields_m
			&=
				\begin{cases}
					\{V_i\mid 1\leq i\leq d\},& m=0,\\
					\{[V_i,U]\mid U\in\vecFields_{m-1},0\leq i\leq d\},& m\geq 1,
				\end{cases}
			&
			\vecFields
			&=
				\bigcup_{m=0}^\infty
					\vecFields_m.
	\end{align*}
	For $x\in\RealNum^n$, $\vecFields_m(x)$ (resp.\,$\vecFields(x)$) stands for the subset of $\RealNum^n$
	obtained by plugging $x$ into the vector field of $\vecFields_m$ (resp.\,$\vecFields$).
	We say that the vector fields $V_0,V_1,\dots,V_d$ satisfy
	the H{\"o}rmander condition at $x\in\RealNum^n$
	if $\vecFields(x)$ linearly spans $\RealNum^n$.
\end{definition}
We assume the following:
\begin{assumption}\label{ass_HorCon}
	$V_0,V_1,\dots,V_d$ satisfy the H\"{o}rmander condition at the initial point $a\in\RealNum^n$.
\end{assumption}
It is known that, under \aref{ass_HorCon},
the law of the solution $y_t$ has a density $p_t(a, a')$
with respect to the Lebesgue measure on $\RealNum^n$
for any $t>0$ (see \cite{HairerPillai2013,CassHairerLittererTindel2015}).
Hence, for any Borel subset $A \subset \RealNum^n$, $\prob(y_t (a) \in A) = \int_A p_t(a, a')\,da'$.

Let $\CM = \CM^\Hurst$ be the Cameron-Martin space of fBm on the time interval $[0,1]$.
For every $\Hurst \in (1/4, 1/2]$, there exists $q \in [1,2)$ such that
every $\gamma \in \CM$ is continuous and of finite $q$-variation.
For $\gamma \in \CM$, we denote by $\phi^0_t =\phi^0_t (\gamma)$
the solution of the following Young ordinary differential equation (ODE):
\begin{align}\label{eq_1526966836}
	d\phi^0_t
	=
		\sum_{i=1}^d
			V_i ( \phi^0_t)\,
			d\gamma^i_t
	 \qquad
	 \text{with}
	 \qquad
	 \phi^0_0=a.
\end{align}
Set, for $a' \neq a$, $K_a^{a'} = \{ \gamma \in \CM \mid \phi^0_1(\gamma) =a'\}$.
We only consider the case where $K_a^{a'} $ is not empty.
From goodness of the rate function in Schilder-type large deviation
for fractional Brownian rough path,
it follows that
$\inf\{ \|\gamma\|_\CM \mid \gamma \in K_a^{a'}\}
= \min\{ \|\gamma\|_\CM \mid \gamma \in K_a^{a'}\}$.
Now we introduce the following assumption:
\begin{assumption}\label{ass_minimizer}
	$\bar{\gamma} \in K_a^{a'}$ which minimizes $\CM$-norm exists uniquely.
\end{assumption}
In what follows, $\bar{\gamma}$ denotes the minimizer in \aref{ass_minimizer}.
Note that the map
$
	\gamma\in\CM\hookrightarrow C^{\var{q}}([0,1];\RealNum^d)
	\mapsto
		\phi^0_1(\gamma)\in\RealNum^n
$
is Fr\'{e}chet differentiable;
$D\phi^0_1(\gamma)$ stands for the Fr\'{e}chet derivative
and is expressed by
\begin{align}\label{eq_1526970364}
	\langle D[\phi^0_1(\gamma)]^k,h\rangle_\CM
	=
		\sum_{i=1}^d
			\int_0^1
				[
					J_1(\gamma)
					K_s(\gamma)
					\sigma(\phi^0_s(\gamma))
				]_i^k\,
				dh^i_s,
\end{align}
where $[\bullet]^k$ and $[\bullet]_i^k$ are the $k$th component of the vector
and the $(k,i)$-component of the matrix, respectively,
and $J(\gamma)$ and $K(\gamma)$ are the Jacobi process of $\phi^0(\gamma)$ and its inverse, respectively.
We define the deterministic Malliavin covariance matrix $Q(\gamma)=(Q(\gamma)_{kl})_{1\leq k,l\leq n}$ by
\begin{align}\label{eq_1527058901}
	Q(\gamma)_{kl}
	=
		\langle D[\phi^0_1(\gamma)]^k,D[\phi^0_1(\gamma)]^l\rangle_{\CM^\ast}.
\end{align}
We assume that $Q$ is non-degenerate at the minimizer $\bar{\gamma}$:
\begin{assumption}\label{ass_DeterministicNondeg}
	There is a positive constant $c$ such that $Q(\bar{\gamma})\geq c\idMat$,
	where $\idMat$ stands  for the identity matrix.
\end{assumption}

Finally, we assume that $\|\bullet\|_\CM^2/2$ is not so degenerate at $\bar{\gamma}$ in the following sense:
\begin{assumption}\label{ass_hessian}
	At $\bar{\gamma}$, the Hessian of the functional $K_a^{a'} \ni \gamma \mapsto \| \gamma \|_\CM^2/2$
	is strictly positive in the quadratic form sense.
	More precisely,
	if $(- \epsilon_0, \epsilon_0) \ni u \mapsto f(u) \in K_a^{a'}$ is a smooth curve in $K_a^{a'}$
	such that $f(0) = \bar{\gamma}$ and $f'(0)\neq 0$, then
	$(d/du)^2 \vert_{u=0} \| f(u) \|_\CM^2 /2 > 0$.
\end{assumption}

\begin{remark}
By a standard argument,
\aref{ass_DeterministicNondeg} is equivalent to
the surjectivity of the tangent map
$D \phi^0_1(\bar{\gamma}) \colon \CM \to \RealNum^n$.
By the implicit function theorem, this implies that
$K_a^{a'}$ has a Hilbert-manifold structure near $\bar{\gamma}$.
Therefore, ``Hessian at $\bar{\gamma}$"
and ``a smooth curve near $\bar{\gamma}$" make sense.
\end{remark}

\subsection{Index sets}\label{sec_1529915726}

In this subsection we introduce several index sets for the exponent of the small time parameter $t>0$,
which will be used in the asymptotic expansion.
Unfortunately, these index sets are not
subsets of $\NaturalNum= \{0,1,2,\ldots\}$
and are rather complicated in general.
(However, they are discrete subsets of $({\mathbb Z}+ \Hurst^{-1}{\mathbb Z}) \cap [0,\infty)$
with the minimum $0$.)

Set
\begin{align*}
	\Lambda_1 = \left\{ n_1 + \frac{n_2}{\Hurst} \,\middle|\, n_1, n_2 \in \NaturalNum \right\}.
\end{align*}
We denote by $0=\kappa_0 <\kappa_1 < \kappa_2 < \cdots$ all the elements of $\Lambda_1$ in increasing order.
Several smallest elements are explicitly given as follows: for $1/3 <\Hurst <1/2$,
\begin{align*}
	\kappa_1
	&=
		1,
	&
	\kappa_2
	&=
		2,
	&
	\kappa_3
	&=
		\frac{1}{\Hurst},
	&
	\kappa_4
	&=
		3,
	&
	\kappa_5
	&=
		1+\frac{1}{\Hurst},
	&
	\kappa_6
	&=
		4,
	\dots
\end{align*}
and, for $1/4<\Hurst<1/3$,
\begin{align*}
	\kappa_1
	&=
		1,
	&
	\kappa_2
	&=
		2,
	&
	\kappa_3
	&=
		3,
	&
	\kappa_4
	&=
		\frac{1}{\Hurst},
	&
	\kappa_5
	&=
		4,
	&
	\kappa_6
	&=
		1+\frac{1}{\Hurst},
	\dots.
\end{align*}
We also set
\begin{gather*}
	\Lambda_2
	=
		\{ \kappa -1 \mid \kappa \in \Lambda_1 \setminus \{\kappa_0\} \}
	=
		\left\{
			0, 1, \frac{1}{\Hurst} -1, 2, \frac{1}{\Hurst}, 3,\dots
		\right\},
				\\
	\Lambda'_2
	=
		\{ \kappa -2 \mid \kappa \in \Lambda_1 \setminus \{\kappa_0, \kappa_1\} \}
	=
		\left\{
			0,
			\frac{1}{\Hurst} -2,
			1,
			\frac{1}{\Hurst}-1,
			2,
			\dots
		\right\}.
\end{gather*}
Note that in the above explicit expression of
these two index sets the elements are not sorted
in increasing order when $1/4 <\Hurst \le 1/3$.
Next we set
\begin{gather*}
	\Lambda_3
	=
		\{
			a_1+a_2+\cdots+a_m
			\mid
				\text{$m \in \NaturalNum_+$ and $a_1 ,\ldots, a_m \in \Lambda_2$}
	\},\\
	\Lambda'_3
	=
		\{
			a_1+a_2+\cdots+a_m
			\mid
			\text{$m \in \NaturalNum_+$ and $a_1 ,\ldots, a_m \in \Lambda'_2$}
	\},
\end{gather*}
where $\NaturalNum_+=\{1,2,\dots\}$.
In the sequel, $0=\nu_0 <\nu_1<\nu_2 <\cdots$ and $0=\rho_0 <\rho_1<\rho_2 <\cdots$ stand for
all the elements of $\Lambda_3$ and $\Lambda'_3$ in increasing order, respectively.
Finally,
\begin{align*}
	\Lambda_4
	=
		\Lambda_3+\Lambda'_3
	=
		\{ \nu + \rho \mid \nu\in \Lambda_3 , \rho \in \Lambda'_3\}.
\end{align*}
We denote by $0 = \lambda_0 < \lambda_1 < \lambda_2 < \cdots$ all the elements of $\Lambda_4$ in increasing order.

We remark that
when $\Hurst=1/2$ and $\Hurst=1/3$, all these index sets $\Lambda_i, \Lambda'_j$ above are just $\NaturalNum$.

\subsection{Statement of the main result}
Now we state our main theorem.
This is basically analogous to many preceding works
on the standard Brownian motion
such as \cite{BenArous1988b, Watanabe1987}.
However, when $\Hurst \neq 1/2,1/3$ and the drift term exists,
there are some differences.
First, the exponents of $t$ are not (a constant multiple of) natural numbers.
Second, cancellation of ``odd terms" (see e.g. \cite[p.~20 and p.~34]{Watanabe1987})
does not occur in general.
(These phenomena were already observed in \cite{Inahama2016b} and \cite{Inahama2016c} when $\Hurst\in(1/3,1)$.)
Our proof uses the Watanabe distribution theory.
Therefore, the asymptotic expansion is actually
obtained at the level of Watanabe distributions.
\begin{theorem} \label{thm_MAIN}
	Assume $a \neq a'$ and \aref{ass_HorCon}--\aref{ass_hessian}.
	Then, we have the following asymptotic expansion as $t \searrow 0$:
	\begin{align*}
		p_t(a, a')
		\sim
			\exp
				\left(
					-
					\frac{ \|\bar{\gamma}\|_\CM^2}{2t^{2\Hurst}}
				\right)
			\frac{1}{t^{n \Hurst}}
			\left\{
				\alpha_{0}
				+\alpha_{\lambda_1} t^{\lambda_1 \Hurst}
				+\alpha_{\lambda_2} t^{\lambda_2 \Hurst}
				+\cdots
			\right\}
	\end{align*}
	for a certain positive constant $\alpha_{0}$ and
	certain real constants $\alpha_{\lambda_j}$ $(j =1,2,\dots)$.
	Here, $0=\lambda_0 <\lambda_1<\lambda_2 <\cdots$ are all the elements of
	$\Lambda_4$ in increasing order.
\end{theorem}

The above result improves that in \cite{Inahama2016c}.
The reason is as follows.
First, we work under H\"ormander condition instead of the ellipticity assumption.
Second, the case $1/4 <\Hurst \le 1/3$ is also treated.
Hence, we must calculate the third level rough paths.

\begin{remark}\label{RemDriftZero}
If RDE \eqref{eq_RDEdrivenByFBM} has no drift term, i.e.,
$V_0 \equiv 0$, then we can repalce $\Lambda_4$ by $2 \NaturalNum$ in \tref{thm_MAIN}, namely,
\begin{align}\label{eq_181109}
		p_t(a, a')
		\sim
			\exp
				\left(
					-
					\frac{ \|\bar{\gamma}\|_\CM^2}{2t^{2\Hurst}}
				\right)
			\frac{1}{t^{n \Hurst}}
			\left\{
				\alpha_{0}
				+\alpha_{2} t^{2 \Hurst}
				+\alpha_{4} t^{4 \Hurst}
				+\cdots
			\right\}
	\end{align}
as $t \searrow 0$.
The reason is as follows.
First, if $V_0 \equiv 0$, the index sets
$\Lambda_i$, $\Lambda'_j$ $(1 \le i \le 4, 2 \le j \le 3)$
in the sequel can be replaced by $\NaturalNum$.
Second, the odd-numbered terms in the asymptotics of
the generalized Wiener functional are also odd
as generalized Wiener functionals.
Hence, their generalized expectations all vanish.

When $\Hurst =1/2$, \eqref{eq_181109} holds for the same reason.
Thus, we reprove the main result of
Ben Arous \cite{BenArous1988b} via rough path theory.
\end{remark}

\section{Preliminaries}\label{sec_1545285214}

\subsection{Rough path analysis}
In this paper we basically work in Lyons' original framework of rough path theory.
(The only exception is \secref{subsec5.2}, where the controlled
path theory is used.)
We borrow most of notations and terminologies from \cite{LyonsQian2002, LyonsCaruanaLevy2007, FrizVictoir2010}.
Let $p \in [2,4)$ be the roughness constant and let $q \in [1,2)$ be such that $1/p +1/q >1$.
Set $\alpha=1/p$ and let $m\in\NaturalNum_+$ satisfy $\alpha-1/m>0$.
We denote by $\geoRPs{p}(\RealNum^d)$,  $\geoRPsHol{\alpha}(\RealNum^d)$, and $\geoRPsBes{\alpha,m}(\RealNum^d)$
the geometric rough path spaces with $p$-variation, $\alpha$-H{\"o}lder and $(\alpha,m)$-Besov topologies, respectively.
In this paper, $\RP{x}= (\RP{x}^1,\dots,\RP{x}^{\intPart{p}})$ stands for
a generic element in these rough path spaces.
Recall that the $p$-variation, $\alpha$-H{\"o}lder and $(\alpha,m)$-Besov topologies are induced by
\begin{gather*}
	\begin{aligned}
		\|\RP{x}^i\|_{\var{p/i}}
		&=
			\sup_{0=t_0<\cdots<t_K=1}
				\left(
					\sum_{k=1}^K
						|\RP{x}^i_{t_{k-1},t_k}|^{p/i}
				\right)^{i/p},
		&
		\|\RP{x}^i\|_{\Hol{i\alpha}}
		&=
			\sup_{0\leq s<t\leq 1}
				\frac{|\RP{x}^i_{s,t}|}{|t-s|^{i\alpha}},\\
	\end{aligned}\\
	\|\RP{x}^i \|_{\Bes{(i\alpha,m/i)}}
	=
		\left(
			\iint_{0 \leq s < t \leq 1}
				\frac{|\RP{x}^i_{s,t}|^{m/i}}{|t-s|^{1+i\alpha\cdot m/i}}\,
				dsdt
		\right)^{i/m}
	\qquad\quad
	(i=1,\dots,\intPart{p}).
\end{gather*}
Next we introduce homogeneous norms to
$\geoRPs{p}(\RealNum^d)$, $\geoRPsHol{\alpha}(\RealNum^d)$, and $\geoRPsBes{\alpha,m}(\RealNum^d)$
which are consistent with their topology.
More explicitly, they are given by
\begin{gather*}
	\begin{aligned}
		\|\RP{x}\|_{\var{p}}
		&=
			\sum_{i=1}^{\intPart{1/\alpha}}
				\|\RP{x}^i \|_{\var{p/i}}^{1/i},
		&
		\|\RP{x}\|_{\Hol{\alpha}}
		&=
			\sum_{i=1}^{\intPart{1/\alpha}}
				\|\RP{x}^i \|_{\Hol{i\alpha}}^{1/i},
	\end{aligned}\\
	\|\RP{x}\|_{\Bes{(\alpha,m)}}
	=
		\sum_{i=1}^{\intPart{1/\alpha}}
			\|\RP{x}^i \|_{\Bes{(i\alpha,m/i)}}^{1/i}.
\end{gather*}
We denote by $\|\RP{x}^i\|_{\var{p/i};[s,t]}$, $\|\RP{x}\|_{\var{p};[s,t]}$, etc.~
the norms restricted on the subinterval $[s,t]$.

For a rough path $\RP{x}$, we write $x_t=\RP{x}^1_{0,t}$ as usual.
For $x \in C_0^{\var{r}} ([0,1]; \RealNum^d)$ with $r\in[1,2)$,
we denote the natural lift of $x$ (i.e., the smooth rough path lying above $x$)
by the corresponding boldface letter $\RP{x}$.
Note that, for $x\in C_0^{\var{r}} ([0,1];\RealNum^d)$ and
$y\in C_0^{\var{r}} ([0,1];\RealNum^e)$,
$(\RP{x}, \RP{y}) \in \geoRPs{p}(\RealNum^{d+e})$ stands for the natural lift of $(x,y)\in C_0^{\var{r}} ([0,1];\RealNum^{d+e})$,
not for the pair $(\RP{x}, \RP{y}) \in \geoRPs{p}(\RealNum^{d}) \times \geoRPs{p}(\RealNum^{e})$.
In a similar way, for $\RP{x} \in \geoRPs{p}(\RealNum^{d})$
and $h \in C_0^{\var{q}}(\RealNum^e)$ with $1/p+1/q>1$,
$(\RP{x}, \RP{h})\in \geoRPs{p}(\RealNum^{d+e})$ stands for the Young pairing.
For $\RP{x} \in \geoRPs{p}(\RealNum^{d})$
and $h \in C_0^{\var{q}}(\RealNum^d)$ with $1/p+1/q>1$,
$\RP{x}+\RP{h}=\tau_h(\RP{x})\in\geoRPs{p}(\RealNum^d)$ stands for the Young translation.
(See \cite[Section~9.4]{FrizVictoir2010}.)
We denote by $c\RP{x}$ the dilation of a rough paths $\RP{x}$
by $c\in\RealNum$.
These notations may be somewhat misleading.
But, they make many operations intuitively clear and easy to understand
when we treat rough paths over a direct sum of many vector spaces.

Note that the spaces $\geoRPs{p}(\RealNum^d)$, $\geoRPsHol{\alpha}(\RealNum^d)$, and $\geoRPsBes{\alpha,m}(\RealNum^d)$
enjoy the next continuous embeddings.
We have
$
	\geoRPsBes{\alpha,m}(\RealNum^d)
	\hookrightarrow
		\geoRPs{p}(\RealNum^d)
$
and
\begin{align}\label{eq_1545292286}
	\|\RP{x}\|_{\var{p};[s,t]}
	\leq
		c
		\|\RP{x}\|_{\Bes{(\alpha,m)}}
		(t-s)^{\alpha -\frac{1}{m}}
	\quad
	\quad
	(\RP{x}\in\geoRPsBes{\alpha,m}(\RealNum^d))
\end{align}
from \cite[Corollary~A.3]{FrizVictoir2010}.
(The continuity of this embedding is not explicitly written in  \cite{FrizVictoir2010},
but can be shown by standard and slightly lengthy argument.)
For $\alpha<\beta\leq 1/\intPart{1/\alpha}$,
a direct computation implies
Besov-H{\"o}lder embedding
$
	\geoRPsHol{\beta}(\RealNum^d)
	\hookrightarrow
		\geoRPsBes{\alpha,m}(\RealNum^d)
$.

For a control function $\omega$ in the sense of \cite[p.~16]{LyonsQian2002},
we write $\bar{\omega} := \omega (0,1)$.
For any $\RP{x} \in \geoRPs{p}(\RealNum^{d})$,
\begin{align}\label{nat_control.def}
	\omega_{\RP{x}} (s,t)
	:=
		\sum_{i=1}^{\intPart{p}}
			\|\RP{x}^i\|_{\var{p/i};[s,t]}^{p/i}
	\qquad\qquad
	(0 \leq s \leq t \leq 1)
 \end{align}
defines a control function.
(This control function
 is equivalent to the one defined by Carnot-Carateodory metric.)
Similarly,
we set $\omega_{h} (s,t) := \|h \|^q_{\var{q}; [s,t]}$
for $h \in C_0^{\var{q}} ([0,1]; \RealNum^e)$.

For $\delta >0$ and $\RP{x} \in \geoRPs{p}(\RealNum^{d})$,
set $\tau_0 (\delta) =0$ and
\begin{align*}
	\tau_{i+1}(\delta)
	=
		\inf
			\{
				t\in (\tau_i(\delta),1]
				\mid
				\omega_{\RP{x}} ( \tau_i (\delta) ,t)
				\geq \delta
			\}
		\wedge 1
		\qquad
		(i=1,2,\dots)
\end{align*}
Define
\begin{align}\label{eq_1529919044}
	N_{\delta} (\RP{x})
	=
		\sup \{ i \in \NaturalNum \mid \tau_i (\delta) <1 \}.
\end{align}
Superadditivity of $\omega_{\RP{x}} $ yields
$\delta N_{\delta} (\RP{x}) \leq \overline{ \omega_{\RP{x}} }$.
This quantity \eqref{eq_1529919044} was first studied by Cass, Litterer, and Lyons \cite{CassLittererLyons2013}.

Let $\RP{x}\in\geoRPs{p}(\RealNum^{d+1})$.
We consider an RDE
\begin{align}\label{eq_1526972098}
	dy_t
	=
		\sum_{i=0}^d
			V_i(y_t)\,
			dx^i_t
	\qquad
	\text{with}
	\qquad
	y_0
	=
		a.
\end{align}
It is well-known that the solution map
$
	\geoRPs{p}(\RealNum^{d+1})
	\ni
	\RP{x}
	\mapsto
		\RP{y}=(\RP{y}^1,\dots,\RP{y}^{\intPart{p}})
	\in
		\geoRPs{p}(\RealNum^{n})
$
is continuous.
Set $\Phi(\RP{x})=\RP{y}$ and $y_t=a+\RP{y}^1_{0,t}$.
The Jacobi process
$
	J
	=
		(
			(
				J^{kl}_t
				=
					\frac{dy^k_t}{da^l}
			)_{t\in[0,1]}
		)_{1\leq k,l\leq n}
$
and
its inverse
$
	K=(K^{kl})_{1\leq k,l\leq n}
$
exist
and satisfy
\begin{align}
	\label{eq_1539835811}
	dJ_t
	&=
		+
		\sum_{i=0}^d
			\nabla V_i(y_t)
			J_t\,
			dx^i_t
	\qquad
	\text{with}
	\qquad
	J_0
	=
		\idMat,\\
	\label{eq_1539835827}
	dK_t
	&=
		-
		\sum_{i=0}^d
			K_t
			\nabla V_i(y_t)\,
			dx^i_t
	\qquad
	\text{with}
	\qquad
	K_0
	=
		\idMat.
\end{align}
Here, $\idMat$ stands for the identity matrix with size $n$.
It is well-known that the system of RDEs
\eqref{eq_1526972098}--\eqref{eq_1539835827}
is globally well-posed.
In particular, the map $\RP{x}\mapsto (\RP{y},\RP{J},\RP{K})$
is continuous.

Here, we make a remark on continuity of the solution map.
Let $\lambda$ be a one-dimensional path defined by $\lambda_t=t$.
The solution map $\RP{x}\mapsto (\RP{y},\RP{J},\RP{K})$ is continuous
on $\geoRPs{p}(\RealNum^{d+1})$ as stated above.
In addition, $\geoRPs{p}(\RealNum^d)\times\RealNum\langle\lambda\rangle$
is embedded in $\geoRPs{p}(\RealNum^{d+1})$ continuously.
Hence, the composition $(\RP{x},c\lambda)\mapsto (\RP{x},c\RP{\lambda})\mapsto  (\RP{y},\RP{J},\RP{K})$ of the two maps
is continuous on $\geoRPs{p}(\RealNum^d)\times\RealNum\langle\lambda\rangle$.
This continuity plays important role in \secref[sec_1542002443]{subsec.ldp},
in which we will use results on large deviation for fractional Brownian rough path.

\subsection{RDEs driven by fBm}
In this paper, we consider the case $x=(w,\lambda)$,
where $w$ is an fBm with the Hurst parameter $1/4<\Hurst\leq 1/2$
and $\lambda$ is a one-dimensional path defined by $\lambda_t=t$.
Let $p$ be larger than but sufficiently close to $1/\Hurst$.
Since $w$ admits a canonical lift $\RP{w}\in\geoRPs{p}(\RealNum^d)$ (see \cite{CoutinQian2002}),
we can deal with \eqref{eq_RDEdrivenByFBM} in the framework of rough path analysis
and obtain $\RP{y}=\Phi(\RP{w},\RP{\lambda})$.
Next we introduce a scaled (and shifted) RDE of \eqref{eq_RDEdrivenByFBM}.
For every $0<\epsilon<1$ and $\gamma\in\CM$,
we consider a scaled RDE
\begin{align}\label{eq_ScaledRDEDrivenByFBm}
	dy^\epsilon_t
	=
		\sum_{i=1}^d
			V_i(y^\epsilon_t)\,
			d(\epsilon w)^i_t
		+
		V_0(y^\epsilon_t)\,
		d(\epsilon^{1/\Hurst}t)
	\qquad
	\text{with}
	\qquad
	y_0
	=
		a,
\end{align}
and a scaled-shifted RDE
\begin{align}\label{eq_ScaledShiftedRDEDrivenByFBm}
	d\tilde{y}^\epsilon_t
	=
		\sum_{i=1}^d
			V_i(\tilde{y}^\epsilon_t)\,
			d(\epsilon w+\gamma)^i_t
		+
		V_0(\tilde{y}^\epsilon_t)\,
		d(\epsilon^{1/\Hurst}t)
	\qquad
	\text{with}
	\qquad
	y_0
	=
		a.
\end{align}
The solutions are respectively given by
\begin{align*}
	\RP{y}^\epsilon
	&=
		\Phi(\epsilon\RP{w},\epsilon^{1/\Hurst}\RP{\lambda}),
	&
	\tilde{\RP{y}}^\epsilon
	&=
			\Phi(\epsilon\RP{w}+\RP{\gamma},\epsilon^{1/\Hurst}\RP{\lambda}).
\end{align*}
It is known that the processes
$(y^{\epsilon}_t)$ and $(y_{\epsilon^{1/\Hurst}t})$ have the same law.
We denote the Jacobi processes and its inverses of $y^\epsilon$ and $\tilde{y}^\epsilon$
by $J^\epsilon$, $K^\epsilon$, $\tilde{J}^\epsilon$ and $\tilde{K}^\epsilon$.
(In the next subsection, a brief summary of the Cameron-Martin space $\CM$
and the Young translation by $\gamma \in \CM$ will be given.)

We remark that everything in this subsection still holds
if $p$-variation topology is replaced by $1/p$-H\"older topology.
See \cite{FrizVictoir2010, FrizVictoir2006b}.

\subsection{The Cameron-Martin space and its isometric space}
Next we introduce the Cameron-Martin space $\CM$ associated to a $d$-dimensional fBm on the time interval $[0,1]$.
Let $\mathfrak{E}_1$ be the linear span of $\{R(t,\bullet)\}_{t\in[0,1]}$
and define the inner product
\begin{align*}
	\left\langle
		\sum_{k=1}^m
			a_k
			R(s_k,\bullet),
		\sum_{l=1}^n
			b_l
			R(t_l,\bullet)
	\right\rangle_{\mathfrak{E}_1}
	=
		\sum_{k=1}^m
		\sum_{l=1}^n
			a_k
			b_l
			R(s_k,t_l).
\end{align*}
The Cameron-Martin space $\CM$ associated to a $d$-dimensional fBm is given by
the completion of $\mathfrak{E}=\mathfrak{E}_1^d$ with respect to the norm
$
	\|\bullet\|_\CM
	=
		\langle
			\bullet,\bullet
		\rangle_{\CM}
$,
where
$
	\langle
		f,g
	\rangle_{\CM}
	=
		\sum_{i=1}^d
			\langle
				f^i,g^i
			\rangle_{\mathfrak{E}_1}
$
for $f=(f^1,\dots,f^d)^\top$, $g=(g^1,\dots,g^d)^\top\in\mathfrak{E}$.
Note that $\CM$ is unitarily isometric to the first Wiener chaos $\wienerChaos{1}$,
which is the Hilbert space defined
by the $\|\bullet\|_{L^2(\probSp;\RealNum)}$-completion of
the linear span of $(w_t)_{t\in [0,1]}$.
From \cite[Corollary~1]{FrizVictoir2006b},
we see the continuous embedding
$\CM\hookrightarrow C_0^{\var{q}}([0,1];\RealNum^d)$ for $(\Hurst+1/2)^{-1}<q<2$.
According to \cite{FrizGessGulisashviliRiedel2016}, the embedding holds for $q=(\Hurst+1/2)^{-1}$.
(Note that if $p > 1/\Hurst$ is sufficiently close to $1/\Hurst$,
then $1/p + 1/q >1$ holds since $\Hurst>1/4$.
Therefore, the Young translation
by $\gamma \in \CM$ is well-defined on $\geoRPs{p}(\RealNum^d)$.)

Next we introduce another Hilbert space $\tilde{\CM}$, which is isometric to $\CM$.
Let $\tilde{\mathfrak{E}}_1$ be the linear span of $\{\indicator{[0,t]}\}_{t\in [0,1]}$
and define a inner product by
\begin{align*}
	\left\langle
		\sum_{k=1}^m
			a_k
			\indicator{[0,s_k]},
		\sum_{l=1}^n
			b_l
			\indicator{[0,t_l]}
	\right\rangle_{\tilde{\mathfrak{E}}_1}
	=
		\sum_{k=1}^m
		\sum_{l=1}^n
			a_k
			b_l
			R(s_k,t_l).
\end{align*}
We define $\tilde{\CM}$
by the completion of $\tilde{\mathfrak{E}}=\tilde{\mathfrak{E}}_1^d$
with respect to the norm
$
	\|\bullet\|_{\tilde{\CM}}
	=
		\langle
			\bullet,\bullet
		\rangle_{\tilde{\CM}}
$,
where
$
	\langle
		f,g
	\rangle_{\tilde{\CM}}
	=
		\sum_{i=1}^d
			\langle
				f^i,g^i
			\rangle_{\tilde{\mathfrak{E}}_1}
$
for $f=(f^1,\dots,f^d)^\top$, $g=(g^1,\dots,g^d)^\top\in\tilde{\mathfrak{E}}$.
Note that $\tilde{\CM}=I^{1/2-\Hurst}_{1-}(L^2([0,1];\RealNum^d))$ from \cite{DecreusefondUstunel1999} and \cite[p.284]{Nualart2006}.
Here, $I^{1/2-\Hurst}_{1-}$ denotes the fractional integral of order $1/2-\Hurst$.
Hence, the inclusions
$
	C^{\Hol{(\Hurst-\delta)}}([0,1];\RealNum^d)
	\subset
		\tilde{\CM}
	\subset
		L^2([0,1];\RealNum^d)
$
hold for $0<\delta<2(\Hurst-1/4)$.
The linear map $\mathcal{R}=(\mathcal{R}^1,\dots,\mathcal{R}^d)^\top\colon\tilde{\mathcal{E}}\to\CM$ defined by
$\mathcal{R}^i \indicator{[0,t]}=R(t,\bullet)$ extends to a unitary isometry
from $\tilde{\CM}$ to $\CM$.

\begin{remark}\label{rem_1519281960}
	Under the condition $1/4<\Hurst\leq 1/2$,
	we can choose two numbers $1/\Hurst<p<1+\intPart{1/\Hurst}$ and $(\Hurst+1/2)^{-1}<q<2$ with $1/p+1/q>1$.
	For every $A=(A_1,\dots,A_d)\in C^{\var{p}}([0,1];\RealNum^d)$, we can define $\phi_A\in\CM^\ast$ by
	\begin{align*}
		\phi_A(h)
		=
			\sum_{i=1}^d
				\int_0^1
					A_i(s)\,
					dh^i_s,
	\end{align*}
	where the right-hand side is the Young integral since $h\in \CM\subset C^{\var{q}}_0([0,1];\RealNum^d)$.
	By using the isometry between $\CM^\ast$ and $\mathcal{C}_1$,
	the definition of $\tilde{\CM}$
	and the inequality $\|\bullet\|_{L^2([0,1];\RealNum^d)}\leq c\|\bullet\|_{\tilde{\CM}}$ for some positive constant $c$,
	we have
	\begin{align*}
		\|\phi_A\|_{\CM^\ast}^2
		=
			\expect
				\Bigg[
					\Bigg(
						\sum_{i=1}^d
							\int_0^1
								A_i(s)\,
								dw^i_s
					\Bigg)^2
				\Bigg]
		=
			\|A\|_{\tilde{\CM}}^2
		\geq
			c^{-2}
			\|A\|_{L^2([0,1];\RealNum^d)}^2.
	\end{align*}
\end{remark}

\subsection{Large deviations}
Let $(1+\intPart{1/\Hurst})^{-1}<\alpha<\beta<\Hurst$
and $m\in\NaturalNum_+$ satisfy $\alpha-1/m>0$.
Recall that fBm $(w_t)$ admits a natural lift a.s.\ via
the dyadic piecewise linear approximation
and the lift $\RP{w}$ is a random variable taking values in $\geoRPsHol{\beta}(\RealNum^d)$.
For $0<\epsilon<1$,
the law of $\epsilon \RP{w}$ on this space
is denoted by $\nu_{\epsilon} = \nu^H_{\epsilon}$.
Note that the lift of Cameron-Martin space $\CM$
is contained in $\geoRPsHol{\beta} (\RealNum^d)$.
Moreover, as $\epsilon \searrow 0$,
Schilder-type large deviations hold
 for
 $\{\nu_{\epsilon}\}_{\epsilon>0}$.
(See \cite{FrizVictoir2007} and \cite[Sections~13.6 and~15.7]{FrizVictoir2010}.)
Because of the Besov-H\"older embedding mentioned above,
these properties also hold in $\geoRPsBes{\alpha,m}(\RealNum^d)$.
As usual, the good rate function $\mathcal{I}$ is given as follows:
$\mathcal{I} (\RP{x})= \|h\|_\CM^2 /2$ if $\RP{x}$ is the lift of some $h\in\CM$
and $\mathcal{I} (\RP{x})= \infty$ if otherwise.

Next, set $\hat\nu_{\epsilon} = \nu_{\epsilon} \otimes \delta_{ \epsilon^{1/\Hurst} \lambda }$,
where $\lambda$ is a one-dimensional path defined by $\lambda_t=t$
and $\otimes$ stands for the product of probability measures.
This measure is supported on
$\geoRPsBes{\alpha,m}(\RealNum^d) \times \RealNum \langle \lambda \rangle$.
The Young pairing map
$
	\geoRPsBes{\alpha,m}(\RealNum^d) \times \RealNum \langle \lambda \rangle
	\to
		\geoRPsBes{\alpha,m} (\RealNum^{d+1})
$
is continuous.
The law of this map under $\hat\nu_{\epsilon}$ is
the law of $(\epsilon \RP{w}, \epsilon^{1/\Hurst} \RP{\lambda})$,
the Young pairing of $\epsilon \RP{w}$ and $\epsilon^{1/\Hurst} \lambda$.
Define $\hat{\mathcal{I}}(\RP{x},l)=\|h\|_{\CM}^2/2$
if $\RP{x}$ is the lift of some $h\in \CM$ and $l_t \equiv 0$
and define $\hat{\mathcal{I}}(\RP{x},l)=\infty$ if otherwise.
Here, $l$ is a one-dimensional path.
We can easily show that $\{\hat\nu_{\epsilon}\}_{\epsilon>0}$ also satisfies a large deviation principle
as $\epsilon \searrow 0$ with a good rate function $\hat{ \mathcal{I}}$.
We will use this in \lref{lm.ldpcut} below to show that
we may localize on a neighborhood of
the minimizer $\bar{\gamma}$ in order to obtain the asymptotic expansion.
(The existence of $\bar{\gamma}$ is assumed in \aref{ass_minimizer}.)

\subsection{The Young translation}\label{sec_1565918229}
In this subsection, we show that the Young translation is well-defined and continuous.
In this subsection, we fix $1/4<\Hurst\leq 1/2$.
Let $(1+\intPart{1/\Hurst})^{-1}<\alpha<\Hurst$
and choose $m\in\NaturalNum_+$ such that $\Hurst-\alpha>2/m$.
Set $p=1/\alpha$ and $q=(\Hurst+1/2-1/m)^{-1}$.
Then we have $1/p+1/q>1$ and $1/q-1/2-\alpha>1/m$.
Since $(\Hurst+1/2)^{-1}\leq q<2$, \cite[Corollary~1]{FrizVictoir2006b} implies
\begin{gather*}
	\|\gamma\|_{\var{q};[s,t]}
	\leq
		c
		\|\gamma\|_{W^{1/q,2}}
		(t-s)^{\frac{1}{q}-\frac{1}{2}}
	\leq
		c
		\|\gamma\|_\CM
		(t-s)^{\frac{1}{q}-\frac{1}{2}}
	\qquad
	(\gamma\in\CM).
\end{gather*}
Throughout this subsection, $c$ denotes a positive constant
and may change from line to line.

We define the Young translation
$
	\tau
	=
		(
			\tau^1,
			\dots,
			\tau^{\intPart{1/\alpha}}
		)
$
from $\geoRPsBes{\alpha,12m}(\RealNum^d)\times\CM$
to $\geoRPsBes{\alpha,12m}(\RealNum^d)$; for
$\RP{x}\in\geoRPsBes{\alpha,12m}(\RealNum^d)$ and $\gamma\in\CM$,
define
\begin{align*}
	\tau_\gamma(\RP{x})^1_{s,t}
	&=
		\RP{x}^1_{s,t}
		+\RP{\gamma}^1_{s,t},\\
	\tau_\gamma(\RP{x})^2_{s,t}
	&=
		\RP{x}^2_{s,t}
		+A^1_{s,t}
		+A^2_{s,t}
		+\RP{\gamma}^2_{s,t},\\
	\tau_\gamma(\RP{x})^3_{s,t}
	&=
		\RP{x}^3_{s,t}
		+B^1_{s,t}
		+B^2_{s,t}
		+B^3_{s,t}
		+C^1_{s,t}
		+C^2_{s,t}
		+C^3_{s,t}
		+\RP{\gamma}^3_{s,t},
\end{align*}
where
\begin{align*}
	\RP{\gamma}^1_{s,t}
	&=
		\gamma_t-\gamma_s,
	&
	\RP{\gamma}^i_{s,t}
	&=
		\int_s^t
			\RP{\gamma}^{i-1}_{s,u}
			\otimes
			d\gamma_u
	\quad (i=2,3),\\
	A^1_{s,t}
	&=
		\int_s^t
			\RP{x}^1_{s,u}
			\otimes
			d\gamma_u,
	&
	A^2_{s,t}
	&=
		\int_s^t
			\RP{\gamma}^1_{s,u}
			\otimes
			dx_u,\\
	B^1_{s,t}
	&=
		\int_s^t
			\RP{x}^2_{s,u}
			\otimes
			d\gamma_u,
	&
	B^2_{s,t}
	&=
		\int_s^t
		\int_s^v
			\RP{x}^1_{s,u}
			\otimes
			d\gamma_u
			\otimes
			dx_v,\\
	B^3_{s,t}
	&=
		-
		\int_s^t
			\RP{\gamma}^1_{s,u}
			\otimes
			d\RP{x}^2_{u,t},
	&
	C^1_{s,t}
	&=
		-
		\int_s^t
			\RP{x}^1_{s,u}
			\otimes
			d\RP{\gamma}^2_{u,t},\\
	C^2_{s,t}
	&=
		\int_s^t
		\int_s^v
			\RP{\gamma}^1_{s,u}
			\otimes
			dx_u
			\otimes
			d\gamma_v,
	&
	C^3_{s,t}
	&=
		\int_s^t
			\RP{\gamma}^2_{s,u}
			\otimes
			dx_u.
\end{align*}
We should understand the integrals in the Young sense since $1/p+1/q>1$.
At first glance, $B^3$ and $C^1$ look strange.
However, if $x$ and $\gamma$ are smooth
and $\RP{x}$ is the natural lift of $x$, we have
\begin{align*}
	B^3_{s,t}
	=
		\int_s^t
			\RP{\gamma}^1_{s,u}
			\otimes
			dx_u
			\otimes
			(x_t-x_u)
	=
		\int_s^t
		\int_s^v
			\RP{\gamma}^1_{s,u}
			\otimes
			dx_u
			\otimes
			dx_v,
\end{align*}
In the first equality, we used
$\frac{d\RP{x}^2_{u,t}}{du}=-\frac{dx_u}{du}\otimes (x_t-x_u)$.
Since a similar equality for $C^3_{s,t}$ holds,
we see $B^3$ and $C^1$ are defined appropriately.
Then we see the next proposition.
\begin{proposition}\label{prop_1530162489}
	Let $\Hurst$, $\alpha$ and $m$ be as above.
	Then, the Young translation
	$
		\tau
		\colon
		\geoRPsBes{\alpha,12m}(\RealNum^d)\times\CM
		\to
		\geoRPsBes{\alpha,12m}(\RealNum^d)
	$
	is well-defined	and continuous.
	Moreover, restricted to every bounded subset of the domain,	$\tau$ is Lipschitz continuous.
\end{proposition}

\begin{proof}
	Let $0\leq s\leq u<v\leq t\leq 1$.
	Recall that $\RP{x}^i_{s,\cdot}$ and $\RP{x}^i_{\cdot,t}$ have finite $p$-variation
	and that
	$\|\RP{x}^i_{s,\cdot}\|_{\var{p};[u,v]}$ and $\|\RP{x}^i_{\cdot,t}\|_{\var{p};[u,v]}$ are bounded above by
	$
		c
		\|\RP{x}\|_{\var{p};[s,t]}^{i-1}
		\|\RP{x}\|_{\var{p};[u,v]}
	$
	for $i=1,2,3$.

	We show well-definedness of the map $\tau$.
	Since $1/q+1/q>1$, we can define $\RP{\gamma}^2$ in the Young sense
	and see that it is of finite $q$-variation and satisfies
	\begin{gather*}
		\|\RP{\gamma}^2_{s,\cdot}\|_{\var{q};[u,v]},\,\,
		\|\RP{\gamma}^2_{\cdot,t}\|_{\var{q};[u,v]}
		\leq
			c
			\|\gamma\|_{\var{q};[s,t]}
			\|\gamma\|_{\var{q};[u,v]}.
	\end{gather*}
	In this estimate, we used \cite[Theorem~6.8]{FrizVictoir2010}.
	By the same reason, $A^1$ and $A^2$ are well-defined in the Young sense
	and they are of finite $q$ and $p$-variation, respectively;
	more precisely, it holds that
	\begin{gather*}
		\|A^1_{s,\cdot}\|_{\var{q};[u,v]}
		\leq
			c
			\|x\|_{\var{p};[s,t]}
			\|\gamma\|_{\var{q};[u,v]},\\
		\|A^2_{s,\cdot}\|_{\var{p};[u,v]}
		\leq
			c
			\|\gamma\|_{\var{q};[s,t]}
			\|x\|_{\var{p};[u,v]}.
	\end{gather*}

	Note that
	$
		B^2_{s,t}
		=
			\int_s^t
				A^1_{s,r}
				\otimes
				dx_r
	$
	and
	$
		C^2_{s,t}
		=
			\int_s^t
				A^2_{s,r}
				\otimes
				d\gamma_r
	$
	can be defined and
	\begin{align*}
		\|B^2_{s,\cdot}\|_{\var{p};[u,v]}
		&\leq
			c
			\|A^1_{s,\cdot}\|_{\var{q};[s,t]}
			\|x\|_{\var{p};[u,v]},\\
		\|C^2_{s,\cdot}\|_{\var{q};[u,v]}
		&\leq
			c
			\|A^2_{s,\cdot}\|_{\var{p};[s,t]}
			\|\gamma\|_{\var{q};[u,v]}.
	\end{align*}
	By the same reason, we see that $B^1_{s,t}$, $B^3_{s,t}$, $C^1_{s,t}$, and $C^3_{s,t}$ are well-defined and
	\begin{align*}
		\|\RP{\gamma}^3_{s,\cdot}\|_{\var{q};[u,v]}
		\leq
			c
			\|\RP{\gamma}^2_{s,\cdot}\|_{\var{q};[s,t]}
			\|\gamma\|_{\var{q};[u,v]},\\
		\|B^1_{s,\cdot}\|_{\var{q};[u,v]}
		\leq
			c
			\|\RP{x}^2_{s,\cdot}\|_{\var{p};[s,t]}
			\|\gamma\|_{\var{q};[u,v]},\\
		\|B^3_{s,\cdot}\|_{\var{p};[u,v]}
		\leq
			c
			\|\gamma\|_{\var{q};[s,t]}
			\|\RP{x}^2_{\cdot,t}\|_{\var{p};[u,v]},\\
		\|C^1_{s,\cdot}\|_{\var{q};[u,v]}
		\leq
			c
			\|x\|_{\var{p};[s,t]}
			\|\RP{\gamma}^2_{\cdot,t}\|_{\var{q};[u,v]},\\
		\|C^3_{s,\cdot}\|_{\var{p};[u,v]}
		\leq
			c
			\|\RP{\gamma}^2_{s,\cdot}\|_{\var{q};[s,t]}
			\|x\|_{\var{p};[u,v]}.
	\end{align*}
	By recalling upper bounds of
	$\|\gamma\|_{\var{q};[u,v]}$,
	$\|\RP{\gamma}^2_{s,\cdot}\|_{\var{q};[u,v]}$,
	$\|\RP{\gamma}^2_{\cdot,t}\|_{\var{q};[u,v]}$,
	$\|x\|_{\var{q};[u,v]}$,
	$\|\RP{x}^2_{s,\cdot}\|_{\var{q};[u,v]}$,
	and $\|\RP{x}^2_{\cdot,t}\|_{\var{q};[u,v]}$,
	we have
	\begin{align*}
		|\RP{\gamma}^2_{s,t}|
		&\leq
			c\|\gamma\|_{\var{q};[s,t]}^2
		\leq
			c\|\gamma\|_\CM^2
			(t-s)^{2(\frac{1}{q}-\frac{1}{2})},\\
		|A^1_{s,t}|,|A^2_{s,t}|
		&\leq
			c
			\|\RP{x}\|_{\var{p};[s,t]}
			\|\gamma\|_{\var{q};[s,t]}\\
		&\leq
			c
			\|\RP{x}\|_{\Bes{(\alpha,12m)}}
			\|\gamma\|_\CM
			\cdot
			(t-s)^{(\alpha-\frac{1}{12m})+(\frac{1}{q}-\frac{1}{2})}.
	\end{align*}
	Here, we used \eqref{eq_1545292286}.
	Moreover, we obtain
	\begin{align*}
		|B^1_{s,t}|,|B^2_{s,t}|,|B^3_{s,t}|
		&\leq
			c
			\|\RP{x}\|_{\var{p};[s,t]}^2
			\|\gamma\|_{\var{q};[s,t]}\\
		&\leq
			c
			\|\RP{x}\|_{\Bes{(\alpha,12m)}}^2
			\|\gamma\|_\CM
			\cdot
			(t-s)^{2(\alpha-\frac{1}{12m})+(\frac{1}{q}-\frac{1}{2})},\\
		|C^1_{s,t}|,|C^2_{s,t}|,|C^3_{s,t}|
		&\leq
			c
			\|\RP{x}\|_{\var{p};[s,t]}
			\|\gamma\|_{\var{q};[s,t]}^2\\
		&\leq
			c
			\|\RP{x}\|_{\Bes{(\alpha,12m)}}
			\|\gamma\|_\CM^2
			\cdot
			(t-s)^{(\alpha-\frac{1}{12m})+2(\frac{1}{q}-\frac{1}{2})}.
	\end{align*}
	Note
	\begin{gather*}
		\left(
			\alpha-\frac{1}{12m}
		\right)
		+
		\left(
			\frac{1}{q}-\frac{1}{2}
		\right)
		=
			2\alpha
			+
			\left(
				\frac{1}{q}-\frac{1}{2}-\alpha-\frac{1}{12m}
			\right)
		>
			2\alpha,\\
		2
		\left(
			\alpha-\frac{1}{12m}
		\right)
		+
		\left(
			\frac{1}{q}-\frac{1}{2}
		\right)
		=
			3\alpha
			+
			\left(
				\frac{1}{q}-\frac{1}{2}-\alpha-\frac{2}{12m}
			\right)
		>
			3\alpha,\\
		\left(
			\alpha-\frac{1}{12m}
		\right)
		+
		2
		\left(
			\frac{1}{q}-\frac{1}{2}
		\right)
		=
			3\alpha
			+
			2
			\left(
				\frac{1}{q}-\frac{1}{2}-\alpha-\frac{1}{2\cdot 12m}
			\right)
		>
			3\alpha,\\
		\frac{1}{q}-\frac{1}{2}
		=
			\alpha
			+
			\left(
				\frac{1}{q}-\frac{1}{2}-\alpha
			\right)
		 >
		 	\alpha
	\end{gather*}
	Hence, $A^i$ and $\RP{\gamma}^2$ are of finite $2(\alpha+\delta)$-H{\"o}lder norm for some $\delta>0$
	and therefore they are of finite $(2\alpha,6m)$-Besov norm.
	We also see that $B^i$, $C^i$ and $\RP{\gamma}^3$ are of finite $(3\alpha,4m)$-Besov norm.
	Hence, $\tau$ is well-defined.

	The continuity follows
	from the continuous embedding $C_0^{\var{q}}([0,1];\RealNum^d)$
	and the continuity of the Young integration.
\end{proof}


\section{Moment estimate for Taylor expansion of the Lyons-It{\^o} map}\label{sec_1545285241}

In this section, we prove a Taylor-like expansion for $\tilde{y}^\epsilon$,
which was defined in \eqref{eq_ScaledShiftedRDEDrivenByFBm}.
The base point $\gamma\in\CM$ of the expansion is arbitrary, but fixed.
We will show that $\tilde{y}^\epsilon$ admits the next expansion
\begin{align}\label{eq_1530005180}
	\tilde{y}^\epsilon
	\sim
		\phi^0
		+\epsilon^{\kappa_1}\phi^{\kappa_1}
		+\cdots
		+\epsilon^{\kappa_k}\phi^{\kappa_k}
		+\cdots
\end{align}
as $\epsilon\searrow 0$ in appropriate senses.
Here $\kappa_k\in\Lambda_1$ and $\phi^{\kappa_k}$ is a nice Wiener functional
(for $\Lambda_1$ and $0=\kappa_0 < \kappa_1 <\kappa_2 <\cdots$, see \secref{sec_1529915726}).

Let us fix some notations for fractional order expansion \eqref{eq_1530005180}.
In this section,
$\Hurst$ is not necessarily the Hurst parameter,
but a positive parameter unless otherwise specified.
For notational simplicity, however, $\Hurst \in (0, 1/2]$ is assumed.
%
%
For such $\Hurst$, we define $0=\kappa_0 < \kappa_1 <\kappa_2 <\cdots$ in the same way as \secref{sec_1529915726}.
We will work in $1$-variation topology for a while.
The ODE that corresponds to \eqref{eq_ScaledShiftedRDEDrivenByFBm} leads
\begin{align}\label{eq_1529913945}
	\nonumber
	d\tilde{y}^\epsilon_t
	&=
		\sigma (\tilde{y}^\epsilon_t)\,
		(\epsilon d x_t + dh_t)
		+
		\epsilon^{1/\Hurst}
		b(\tilde{y}^\epsilon_t)\,dt\\
	&=
		\sigma(\tilde{y}^\epsilon_t)
		\epsilon\,
		dx_t
		+
		\bigl[
			\sigma (\tilde{y}^\epsilon_t)\,
			dh_t
			+
			\epsilon^{1/\Hurst}
			b(\tilde{y}^\epsilon_t)\,
			d\lambda_t
		\bigr]
		\quad
		\text{with}
		\quad
		\tilde{y}^\epsilon_0
		=
			a.
\end{align}
Here $x,h\in C_0^{\var{1}}([0,1];\RealNum^d)$ and $\lambda_t=t$.
Next we define $\phi^{\kappa_m}$ appeared in \eqref{eq_1530005180} formally;
for intuitive derivation of $\phi^{\kappa_m}$, see \cite[Section~7]{Inahama2016b}.
By setting $\epsilon=0$ in \eqref{eq_1529913945}, we can easily see that $\phi^0=\phi^0(h)$ satisfies
\begin{align}\label{eq_1530091699}
	d\phi^{0}_t
	&=
		\sigma(\phi^{0}_t)\,
		dh_t
	\quad
	\text{with}
	\quad
	\phi^{0}_0
	=
		a.
\end{align}
An ODE for $\phi^1=\phi^1(x,h)$ is given by
\begin{align}\label{eq_1530091711}
	d\phi^1_t
	-
	\nabla\sigma (\phi^0_t)\langle \phi^1_t, dh_t \rangle
	&=
		\sigma (\phi^0_t)\,
		dx_t
	\quad
	\text{with}
	\quad
	\phi^{1}_0
	=
		0.
\end{align}
For $\kappa_m\geq 2$, $\phi^{\kappa_m}=\phi^{\kappa_m}(x,h)$ satisfies
\begin{multline}\label{eq_1530091405}
	d\phi^{\kappa_m}_t
	-
	\nabla\sigma (\phi^0_t)\langle \phi^{\kappa_m}_t, dh_t \rangle
	=
		\sum_{k=1}^\infty
		\sum_{\kappa_{i_1}+\cdots+\kappa_{i_k}=\kappa_m-1}
			\frac{\nabla^k\sigma(\phi^0_t)}{k!}
			\langle
				\phi^{\kappa_{i_1}},
				\dots,
				\phi^{\kappa_{i_k}};
				dx_t
			\rangle\\
	\begin{aligned}
		&
			+
			\sum_{k=2}^\infty
			\sum_{\kappa_{i_1}+\cdots+\kappa_{i_k}=\kappa_m}
				\frac{\nabla^k\sigma(\phi^0_t)}{k!}
				\langle
					\phi^{\kappa_{i_1}},
					\dots,
					\phi^{\kappa_{i_k}};
					dh_t
				\rangle\\
		&
			+
			\sum_{k=1}^\infty
			\sum_{\kappa_{i_1}+\cdots+\kappa_{i_k}=\kappa_m-1/\Hurst}
				\frac{\nabla^k b(\phi^0_t)}{k!}
				\langle
					\phi^{\kappa_{i_1}},
					\dots,
					\phi^{\kappa_{i_k}}
				\rangle\,
				dt
		\quad
		\text{with}
		\quad
		\phi^{\kappa_m}_0
		=
			0.
	\end{aligned}
\end{multline}
The summations in the first term on the right-hand side is taken over all
$\kappa_{i_1},\dots,\kappa_{i_k}\in\Lambda_1\setminus\{0\}$
such that $\kappa_{i_1}+\cdots+\kappa_{i_k}=\kappa_m-1$ holds.
It is not allowed that $\kappa_{i_j}=0$.
So, the sum is actually a finite sum.
The second and the third terms should be understood in the same way.
An important observation is that the right-hand side of \eqref{eq_1530091405}
does not involve $\phi^{\kappa_{m}}$, but only $\phi^0,\dots,\phi^{\kappa_{m-1}}$.
These ODEs have a rigorous meaning and
can actually be solved by variation of constants method.
So, we inductively define $\phi^{\kappa_{m}}$
as a unique solution of \eqref{eq_1530091699}, \eqref{eq_1530091711} and \eqref{eq_1530091405}.
Finally,
set
\begin{align}\label{eq_1529915801}
	r_\epsilon^{\kappa_{k+1}}
	=
		\tilde{y}^\epsilon
		-
		(
			\phi^0
			+\epsilon^{\kappa_1}\phi^{\kappa_1}
			+\cdots
			+\epsilon^{\kappa_k}\phi^{\kappa_k}
		).
\end{align}
Then,
$
	(x,h)
	\mapsto
	(
		x,
		h,
		\lambda,
		\tilde{y}^\epsilon,
		\phi^{\kappa_0},
		\dots,
		\phi^{\kappa_k},
		r_\epsilon^{\kappa_{k+1}}
	)
$
is continuous from $C_0^{\var{1}}([0,1];\RealNum^d)^2$
to $C_0^{\var{1}}([0,1];(\RealNum^d)^{\oplus 2}\oplus\RealNum\oplus(\RealNum^n)^{k+3})$.

Furthermore, it is known that
this map extends to a continuous map with respect to the rough path topology;
more precisely, for $2 \le p <4$,
$1\le q <2$ with $1/p+1/q>1$,
\begin{multline*}
	\geoRPs{p}(\RealNum^d)
	\times
	C^{\var{q}}_0([0,1];\RealNum^d)
	\ni
		(\RP{x},h)\\
	\mapsto
	(
		\RP{x},
		\RP{h},
		\RP{\lambda},
		\tilde{\RP{y}}^\epsilon,
		\RP{\phi}^0,
	 	\RP{\phi}^{\kappa_{1}},
		\dots,
		\RP{\phi}^{\kappa_{k}},
		\RP{r}_\epsilon^{\kappa_{k+1}}
	)
	\in
	\geoRPs{p}((\RealNum^d)^{\oplus 2}\oplus\RealNum\oplus(\RealNum^n)^{k+3})
\end{multline*}
is locally Lipschitz for any $k$.
Here, \eqref{eq_1529913945} is viewed as an RDE
driven by $(\epsilon \RP{x} +\RP{h}, \RP{\lambda})$
in the same way as in \eqref{eq_ScaledShiftedRDEDrivenByFBm}.

\begin{proposition}\label{pr.map.rp2.5}
Assume $2\le p <4$, $1 \le q <2$ and $1/p+1/q>1$.
Let
	$(\RP{x},h)\in \geoRPs{p}(\RealNum^d)
	\times
	C^{\var{q}}_0([0,1];\RealNum^d)$ and
consider RDE \eqref{eq_1529913945} and keep the same notations as above.
	Then, the following hold.
	\begin{enumerate}
		\item	For any $\rho >0$ and $k=1,2,\ldots$, there exists a positive constants $C=C(\rho, k)$
				which satisfies that
				\begin{align*}
					\| (\RP{\phi}^{\kappa_{k}} )^1 \|_{\var{p}} \leq C (1+ \overline{\omega_{ \RP{x}}}^{1/p})^{\kappa_{k}}.
				\end{align*}
				for any $\RP{x} \in \geoRPs{p}(\RealNum^d)$
				and $h \in C_0^{\var{q}}([0,1];\RealNum^d)$ with $\| h\|_{\var{q}} \leq \rho$.
		\item	For any $\rho_1, \rho_2 >0$ and $k=1,2,\ldots$,
				there exists a positive constants $\tilde{C}= \tilde{C}(\rho_1, \rho_2, k)$, which is independent of $\epsilon$
				and satisfies that
				\begin{align*}
					\| (\RP{r}^{\kappa_{k+1}}_\epsilon)^1 \|_{\var{p}}
					\leq \tilde{C} (\epsilon + \epsilon \overline{\omega_{ \RP{x}}}^{1/p})^{\kappa_{k+1}}
				\end{align*}
				for any $\RP{x} \in \geoRPs{p}(\RealNum^d)$ with
				$
					\overline{\omega_{\epsilon\RP{x}}}^{1/p}
					=
						\epsilon \overline{\omega_{ \RP{x}}}^{1/p}
					\leq
						\rho_1
				$
				and any $h \in C_0^{\var{q}}([0,1];\RealNum^d )$
				with $\|h\|_{\var{q}} \leq \rho_2$.
	\end{enumerate}
\end{proposition}

\begin{proof}
This is essentially shown in \cite{Inahama2010a}.
The only difference is that the index set in  \cite{Inahama2010a}
is $\NaturalNum$
while it is $\Lambda_1$ here.
But, this is a trivial issue.
\end{proof}

Now, we provide an $L^r$-version of the above proposition.
By integrability lemmas of $\overline{\omega_{\RP{x}}}$ and $N_{\delta}(\RP{x})$,
which are found in \cite[Proposition~7]{FrizOberhauser2010} and \cite[Thereom~6.3 and Remark~6.4]{CassLittererLyons2013}, respectively,
the following theorem applies to
fractional Brownian rough path $\RP{w}$ and $h \in \CM$
when $\Hurst \in (1/4, 1/2]$.
More precisely, for every $\Hurst \in (1/4, 1/2]$,
and $p > 1/\Hurst$ that is sufficiently close to $1/\Hurst$,
we can let $\Hurst$ be the Hurst parameter itself
and take $\RP{x} =\RP{w}$ and $h \in \CM$
in \tref{thm_1529920416} below.

\begin{theorem}\label{thm_1529920416}
Assume $2\le p <4$, $1 \le q <2$ and $1/p+1/q>1$.
Let $h \in C_0^{\var{q}}([0,1];\RealNum^d)$
%
and let $\RP{x}$ be a $\geoRPs{p}(\RealNum^d)$-valued random variable
	such that $\overline{\omega_{\RP{x}}}=\omega_{\RP{x}} (0,1) \in \bigcap_{1 \leq r <\infty} L^r$
	and $\exp(N_{\delta}(\RP{x})) \in \bigcap_{1 \leq r <\infty} L^r$ for every $\delta >0$.
We consider RDE \eqref{eq_1529913945}.

	Then, for any $\RP{x}$, $h$ and $k \in \NaturalNum$,
	there exist control functions $\eta_k = \eta_{k,\RP{x}, h }$
	such that the following hold:
	\begin{enumerate}
		\item	$\eta_{k}$ are non-decreasing in $k$, i.e.,
				$\eta_{k,\RP{x}, h}(s,t) \leq \eta_{k+1,\RP{x}, h} (s,t)$
				for all $k, \RP{x}, h, (s,t)$.
		\item	$\overline{\eta_{k,\RP{x}, h} } \in \bigcap_{1 \leq r <\infty} L^r$ for all $k, h$.
		\item	For all $\epsilon \in (0,1]$, $k\in \NaturalNum$, $h$, $\RP{x}$,
				$0 \leq s \leq t \leq 1$, and $1\leq j\leq \intPart{p}$,
				we have
				\begin{align*}
					\bigl|
						\bigl(
							\RP{x},
							\RP{h},
							\tilde{\RP{y}}^\epsilon,
							\RP{\phi}^0,
							\RP{\phi}^{\kappa_1},
							\dots,
							\RP{\phi}^{\kappa_k},
							\epsilon^{-\kappa_{k+1}}\RP{r}_\epsilon^{\kappa_{k+1}}
						\bigr)^j_{s,t}
					\bigr|
					\leq
						\eta_{k,\RP{x}, h} (s,t)^{j/p}.
				\end{align*}
	\end{enumerate}
	In particular, for all $k\in \NaturalNum$ and $h$,
	$\| (\RP{\phi}^{\kappa_{k}})^1\|_{\var{p}} \in \bigcap_{1 \leq r <\infty} L^r$
	and $ \| (\RP{r}_\epsilon^{\kappa_{k+1}})^1 \|_{\var{p}} = O (\epsilon^{\kappa_{k+1}})$ in $L^r$
	for any $1 \leq r <\infty$.
\end{theorem}

\begin{proof}
	One can prove this by combining
	\pref{pr.map.rp2.5} above and computing $N_{\delta}(\RP{x})$
	in the same way as \cite{Inahama2016c}.
	For $k=0$, \eqref{eq_1529913945} and \eqref{eq_1530091699} imply
	\begin{align*}
		\frac{1}{\epsilon}
		dr^1_{\epsilon,t}
		&=
			\sigma(\tilde{y}^\epsilon_t)\,dx_t
			+
			\frac{1}{\epsilon}
			\{\sigma(\tilde{y}^\epsilon_t)-\sigma(\phi^0_t)\}\,
			dh_t
			+
			\epsilon^{1/\Hurst-1}
			b(\tilde{y}^\epsilon_t)\,
			dt.
	\end{align*}
	Since the first term in the right-hand side
	can be interpreted as a rough integration of one-form
	along $(\RP{x},\RP{h},\tilde{\RP{y}}^\epsilon,\RP{\phi}^0)$,
	there exists a control function $\eta_{\RP{x}, h}'$ determined by
	$\omega_{\RP{x}}$ and $\omega_h$ such that
	\begin{align*}
		\bigl|
			\bigl(
				\RP{x},
				\RP{h},
				\tilde{\RP{y}}^\epsilon,
				\RP{\phi}^0,
				\int
					\sigma(\tilde{\RP{y}}^\epsilon)\,
					\RP{x}
			\bigr)^j_{s,t}
		\bigr|
		\leq
			\eta_{\RP{x}, h}' (s,t)^{j/p}
	\end{align*}
	for $1\leq j\leq \intPart{p}$.
	The second and third terms can admit
	\begin{align*}
		\left|
			\frac{1}{\epsilon}
			\int_s^t
				\{\sigma(\tilde{y}^\epsilon_u)-\sigma(\phi^0_u)\}\,
				dh_u
			+
			\epsilon^{1/\Hurst-1}
			\int_s^t
				b(\tilde{y}^\epsilon_u)\,
				du
		\right|
		\leq
			\eta_{\RP{x}, h}'' (s,t)^{1/q},
	\end{align*}
	where $\eta_{\RP{x}, h}''$ is a control function
	determined by $\omega_{\RP{x}}$, $\omega_h$ and $N_{\delta}(\RP{x})$.
	These three equations correspond to (3.16), (3.17) and (3.20) in \cite{Inahama2016c}.
	Hence, the assertion for $k=0$ holds
	from these estimates and a property of Young pairing
	($\eta_{1,\RP{x},h}$ appeared in the assertion consists of
	$\eta_{\RP{x}, h}'$ and $\eta_{\RP{x}, h}''$).
	For detail of the discussion above and proof of the assertion for general $k$,
	we consult on \cite[Section~3]{Inahama2016c}.
\end{proof}

\section{Malliavin calculus for solution of RDE driven by fBm}\label{sec_1545285261}

Our main purpose in this section is to prove
that $y^\epsilon_1$ and $\tilde{y}^\epsilon_1$ are non-degenerate uniformly in the sense of Malliavin in $\epsilon$.
Here, $y^\epsilon$ and $\tilde{y}^\epsilon$ are
solutions to \eqref{eq_ScaledRDEDrivenByFBm} and \eqref{eq_ScaledShiftedRDEDrivenByFBm}, respectively.
Throughout this section, $\gamma\in\CM$ in the definition of $\tilde{y}^\epsilon_t$
in \eqref{eq_ScaledShiftedRDEDrivenByFBm} will be fixed.
%
%

\subsection{Notation and results of this section}
We recall Watanabe's theory of generalized Wiener functionals (Watanabe distributions) in Malliavin calculus.
We borrow the notations used in this paper from Ikeda and Watanabe \cite[V.8-V.10]{IkedaWatanabe1989}.
We also refer to Nualart \cite{Nualart2006}, Shigekawa \cite{Shigekawa2004},
Matsumoto and Taniguchi \cite{MatsumotoTaniguchi2017} and Hu \cite{Hu2017}.

We introduce a measure.
Let $\prob=\prob^\Hurst$ be the law of the fBm with Hurst parameter $\Hurst$.
This is a probability measure on $\probSp= \overline{\CM}$,
which is the closure of the Cameron-Martin space $\CM=\CM^\Hurst$ in $C_0^{\var{p}}([0,1];\RealNum^d)$.
Then, the triple $(\probSp, \CM, \prob)$ is an abstract Wiener space.
We denote by $\SobSp{r}{s}(K)$ the $K$-valued Gaussian-Sobolev space,
where $1<r<\infty$, $s\in\RealNum$ and $K$ is a real separable Hilbert space.
Note that $r$ and $s$ stands for the integrability index and the differentiability index, respectively.
As usual, we set
the spaces $\SobSp{\infty}{}(K)=\bigcap_{s=1}^\infty\bigcap_{1<r<\infty}\SobSp{r}{s}(K)$
and $\tilSobSp{\infty}{}(K)=\bigcap_{s=1}^\infty\bigcup_{1<r<\infty}\SobSp{r}{s}(K)$ of test functions
and the spaces $\SobSp{-\infty}{}(K)=\bigcup_{s=1}^\infty\bigcup_{1<r<\infty}\SobSp{r}{-s}(K)$
and $\tilSobSp{-\infty}{}(K)=\bigcup_{s=1}^\infty\bigcap_{1<r<\infty}\SobSp{n}{-s}(K)$ of Watanabe distributions.
For $F\in \SobSp{r}{s}(K)$, $D^kF$ denotes the $k$th derivative of $F$ for $k\in\NaturalNum_+$.

We set $V^\epsilon_0=\epsilon^{1/\Hurst}V_0$, $V^\epsilon_i=\epsilon V_i$ ($1\leq i\leq d$),
and $\sigma^\epsilon=[V^\epsilon_1,\dots,V^\epsilon_d]$,
which is viewed as an $n\times d$ matrix.
It holds that $\sigma^\epsilon=\epsilon\sigma$.
Set
$
	\tilde{A}^\epsilon(s)
	=
		\tilde{J}^\epsilon_1
		\tilde{K}^\epsilon_s
		\sigma^\epsilon(\tilde{y}^\epsilon_s)
$,
whose size is again  $n\times d$.
Let $\tilde{A}^{\epsilon,k}(s)$ and $\tilde{A}^{\epsilon,k}_i(s)$ denote
the $k$th row and the $(k,i)$-component of $\tilde{A}^\epsilon(s)$, respectively.

Under this setting, we see that $\tilde{y}^\epsilon_1$ is differentiable in the sense of Malliavin
and the derivatives admit good estimate as follows:
\begin{proposition}\label{prop_1529919750}
	We have $\tilde{y}^\epsilon_1\in\SobSp{\infty}{}(\RealNum^n)$
	and
	\begin{align}\label{eq_1526971767}
		\langle D\tilde{y}^{\epsilon,k}_1,h \rangle_{\CM}
		=
			\phi_{\tilde{A}^{\epsilon,k}}(h)
		=
			\sum_{i=1}^d
				\int_0^1
					\tilde{A}^{\epsilon,k}_i(s)\,
					dh^i_s
		\qquad
		(k=1,\dots,n).
	\end{align}
	In addition, for any $m=0,1,2,\ldots$ and $1<r<\infty$,
	there exists a positive constant $c=c_{m,r}$ such that
	\begin{align*}
		\expect[\| D^m \tilde{y}^{\epsilon,k}_1 \|_{ \CM^{\otimes m} }^r]^{1/r}
		\leq
			c
			\epsilon^m
		\qquad
		(k=1,\dots,n).
	\end{align*}
\end{proposition}

\begin{proof}
	Malliavin differentiability of solutions of RDEs driven by Gaussian rough path
	were studied in \cite{Inahama2014}.
	A slight modification of that argument proves this proposition.
\end{proof}

\begin{proposition}\label{pr.asyDinf}
	We have the following asymptotic expansion as $\epsilon \searrow 0$:
	\begin{align*}
		\tilde{y}^\epsilon_1
		\sim
			\phi^0_1
			+\epsilon^{\kappa_1}\phi^{\kappa_1}_1
			+\cdots
			+\epsilon^{\kappa_k}\phi^{\kappa_k}_1
			+\cdots
		\qquad
		\text{in $\SobSp{\infty}{}(\RealNum^n)$}.
	\end{align*}
	This means that for each $k$,
	(i) $\phi^{\kappa_k}_1 \in \SobSp{\infty}{}(\RealNum^n)$
	and (ii) $\SobSp{r}{s}$-norm of $r_{\epsilon, 1}^{\kappa_{k+1}}$
	is $O (\epsilon^{\kappa_{k+1}} )$ for any $1<r<\infty$ and $s \geq 0$.
	Here, $r_{\epsilon, 1}^{\kappa_{k+1}}$ is defined by \eqref{eq_1529915801}.
\end{proposition}
\begin{proof}
	We can show the assertion in the same way as \cite[Proposition~4.3]{Inahama2016c}.
	Indeed, it follows from \tref{thm_1529920416}, \pref{prop_1529919750},
	Meyer's inequality and the fact that $\phi^{\kappa_k}_1$ belongs
	to the $\RealNum^n$-valued inhomogeneous Wiener chaos of order $\lfloor \kappa_k\rfloor$.
\end{proof}

The Malliavin covariance matrices $Q=(Q_{kl})_{1\leq l,k\leq n}$ of $y_1$,
$Q^\epsilon=(Q^\epsilon_{kl})_{1\leq l,k\leq n}$ of $y^\epsilon_1$,
and $\tilde{Q}^\epsilon=(\tilde{Q}^\epsilon_{kl})_{1\leq l,k\leq n}$ of $\tilde{y}^\epsilon_1$
are defined by
\begin{align}\label{eq_1527058858}
	Q_{kl}
	&=
		\langle Dy^k_1,Dy^l_1\rangle_\CM,
	&
	Q^\epsilon_{kl}
	&=
		\langle Dy^{\epsilon,k}_1,Dy^{\epsilon,l}_1\rangle_\CM,
	&
	\tilde{Q}^\epsilon_{kl}
	&=
		\langle D\tilde{y}^{\epsilon,k}_1,D\tilde{y}^{\epsilon,l}_1\rangle_\CM,
\end{align}
respectively.
In this paper we do not express these covariance matrices
as two-parameter Young integrals (cf. \cite[(6.1)]{CassHairerLittererTindel2015}).
In this notation, we will state the following two propositions.

The first one is Kusuoka-Stroock type estimate
for the solution of the scaled RDE.
Although this is not surprising, there seems to be no
literature that actually proves it.

\begin{proposition}\label{prop_1526540975}
	Suppose that \aref{ass_HorCon} holds.
	Then, there are positive constants $c$ and $\mu$
	such that for every $0<\epsilon<1$ and $1<r<\infty$,
	we have
	\begin{align*}
		\expect[|\det Q^\epsilon|^{-r}]^{1/r}
		\leq
			c\epsilon^{-\mu}.
	\end{align*}
	Here, $c=c(r)$ can be chosen independent of $\epsilon$,
	while $\mu$ is independent of both $r$ and $\epsilon$.
\end{proposition}

The second one is about the
uniform non-degeneracy in the sense of Malliavin calculus
of the solution of the scaled shifted RDE
under \aref{ass_DeterministicNondeg}.
The Schilder type large deviations for fractional Brownian
rough path are used in the proof.

\begin{proposition}\label{prop_1526956280}
	Suppose that \aref{ass_HorCon} holds.
	Let $\gamma$ in the definition of $\tilde{y}^{\epsilon}$ in
	\eqref{eq_ScaledShiftedRDEDrivenByFBm}
	satisfy $Q(\gamma)\geq c\idMat$ for some $c>0$.
	Then, for every $1<r<\infty$, we have
	\begin{align*}
		\sup_{0<\epsilon<1}
			\expect[|\det \epsilon^{-2} \tilde{Q}^\epsilon|^{-r}]^{1/r}
		<
			\infty.
	\end{align*}
\end{proposition}
Let $a'\in\RealNum^n$.
Then, we have
\begin{align*}
	\epsilon^{-2} \tilde{Q}^\epsilon_{kl}
	=
		\left\langle
			D
				\left(
					\frac{\tilde{y}^{\epsilon,k}_1-(a')^k}{\epsilon}
				\right),
			D
				\left(
					\frac{\tilde{y}^{\epsilon,l}_1-(a')^l}{\epsilon}
				\right)
		\right\rangle_\CM.
\end{align*}
It follows from \pref{prop_1526956280} and this identity
that $(\tilde{y}^\epsilon_1-a')/\epsilon$ with $\gamma =\bar{\gamma}$
is uniformly non-degenerate in the sense of  Malliavin calculus
under \aref{ass_HorCon}, \aref{ass_minimizer} and \aref{ass_DeterministicNondeg}.

\pref[prop_1526540975]{prop_1526956280} will be proved in \secref[subsec5.2]{sec_1542002443}, respectively.

\subsection{Covariance matrix of solution of scaled RDE driven by fBm}
\label{subsec5.2}
In this subsection, we show \pref{prop_1526540975}.
Set
$
	C^\epsilon
	=
		\int_0^1
			K^\epsilon_s \sigma^\epsilon(y^\epsilon_s)
			\{K^\epsilon_s\sigma^\epsilon(y^\epsilon_s)\}^\top\,
			ds
$.
Then we see \pref{prop_1526540975} from the following three lemmas.

\begin{lemma}\label{lem_1526961113}
	For every $0<\epsilon<1$, we have
	\begin{align*}
		\MinEigenVal(Q^\epsilon)
		\geq
			c
			\MinEigenVal(C^\epsilon)
			\MinEigenVal(J^\epsilon_1(J^\epsilon_1)^\top).
	\end{align*}
	Here, $c$ is a positive constant independent of $\epsilon$.
\end{lemma}

\begin{lemma}\label{lem_1523247628}
	Suppose that \aref{ass_HorCon} holds.
	Then, there are positive constants $c$ and $\mu$
	such that for every $0<\epsilon<1$ and $1<r<\infty$,
	we have
	\begin{align*}
		\expect[\MinEigenVal(C^\epsilon)^{-r}]^{1/r}
		\leq
			c\epsilon^{-\mu}.
	\end{align*}
	Here, $c=c(r)$ can be chosen independent of $\epsilon$,
	while $\mu$ is independent of both $r$ and $\epsilon$.
\end{lemma}

\begin{lemma}[\cite{CassLittererLyons2013}]\label{lem_1523340339}
	Let $p>1/\Hurst$.
	For every $1<r<\infty$, we have
	\begin{align*}
		\sup_{0<\epsilon<1}
			\expect[\|J^\epsilon\|_{C^{\var{p}}([0,1];\RealNum^{n^2})}^r]^{1/r}
		&<
			\infty,
		&
		\sup_{0<\epsilon<1}
		\expect[\|K^\epsilon\|_{C^{\var{p}}([0,1];\RealNum^{n^2})}^r]^{1/r}
		&<
			\infty.
	\end{align*}
\end{lemma}

Now we show \pref{prop_1526540975} by using lemmas above.
\begin{proof}[Proof of \pref{prop_1526540975}]
	From \lref{lem_1526961113}, we see
	\begin{align*}
		\det Q^\epsilon
		\geq
			\MinEigenVal(Q^\epsilon)^n
		\geq
			c^n
			\MinEigenVal(C^\epsilon)^n
			\MinEigenVal(J^\epsilon_1(J^\epsilon_1)^\top)^n.
	\end{align*}
	Hence,
	\begin{align*}
		\expect[\{\det Q^\epsilon\}^{-r}]^{1/r}
		&\leq
			c^{-n}
			\expect
				[
					\MinEigenVal(C^\epsilon)^{-nr}
					\MinEigenVal(J^\epsilon_1(J^\epsilon_1)^\top)^{-nr}
				]^{1/r}\\
		&\leq
			c^{-n}
			\expect[\MinEigenVal(C^\epsilon)^{-2nr}]^{1/(2r)}
			\expect[\MinEigenVal(J^\epsilon_1(J^\epsilon_1)^\top)^{-2nr}]^{1/(2r)}.
	\end{align*}
	Noting $\MinEigenVal(J^\epsilon_1(J^\epsilon_1)^\top)^{-1}=\MaxEigenVal(K^\epsilon_1(K^\epsilon_1)^\top)$
	and applying \lref[lem_1523247628]{lem_1523340339}, we see the assertion.
\end{proof}

Next we  show  \lref{lem_1526961113}.
\begin{proof}[Proof of \lref{lem_1526961113}]
	Set
	$
		A^\epsilon(s)
		=
			\indicator{[0,1]}(s)
			J^\epsilon_1
			K^\epsilon_s
			\sigma^\epsilon(y^\epsilon_s)
	$
	and let $A^{\epsilon,k}(s)$ denote the $k$th row of $A^\epsilon(s)$.
	Then, for every $v=(v_1,\dots,v_n)^\top\in\RealNum^n$, we have
	$
		\langle v,Q^\epsilon v\rangle_{\RealNum^n}
		=
			\|
				\sum_{k=1}^n
					v_k
					Dy^{\epsilon,k}_1
			\|_\CM^2
	$.
	It follows from
	the Riesz representation theorem and \rref{rem_1519281960}
	that
	\begin{align*}
		\left\|
			\sum_{k=1}^n
				v_k
				Dy^{\epsilon,k}_1
		\right\|_\CM^2
		=
			\left\|
			\sum_{k=1}^n
				v_k
				Dy^{\epsilon,k}_1
			\right\|_{\CM^\ast}^2
		\geq
			c
			\left\|
				\sum_{k=1}^n
					v_k
					A^{\epsilon,k}
			\right\|_{L^2([0,1];\RealNum^d)}^2.
	\end{align*}
	Here, $c$ is a positive constant independent of $v$ and $\epsilon$.
	Noting
	$
		|
			\sum_{k=1}^n
				v_k
				A^{\epsilon,k}(s)
		|_{\RealNum^d}^2
		=
			\langle
				v,
				A^\epsilon(s)A^\epsilon(s)^\top v
			\rangle_{\RealNum^n}
	$,
	we have
	\begin{align*}
		\left\|
			\sum_{k=1}^n
				v_k
				A^{\epsilon,k}
		\right\|_{L^2([0,1];\RealNum^d)}^2
		=
			\int_0^1
				\langle
					v,
					A^\epsilon(s)A^\epsilon(s)^\top v
				\rangle_{\RealNum^n}\,
				ds
		=
		\langle
			v,
			J^\epsilon_1C_\epsilon (J^\epsilon_1)^\top v
		\rangle_{\RealNum^n}.
	\end{align*}
	These imply
	\begin{align*}
		\langle v,Q^\epsilon v\rangle_{\RealNum^n}
		\geq
			c
			\langle
				(J^\epsilon_1)^\top v,
				C^\epsilon
				(J^\epsilon_1)^\top v
			\rangle_{\RealNum^n}.
	\end{align*}
	Hence,
	\begin{align*}
		\langle v,Q^\epsilon v\rangle_{\RealNum^n}
		\geq
			c
			\lambda_{\min}(C^\epsilon)
			\langle
				(J^\epsilon_1)^\top v,
				(J^\epsilon_1)^\top v
			\rangle_{\RealNum^n}
		\geq
			c
			\lambda_{\min}(C^\epsilon)
			\lambda_{\min}(J^\epsilon_1(J^\epsilon_1)^\top)
			|v|^2.
	\end{align*}
	The proof has been completed.
\end{proof}

In the rest of this subsection we show \lref{lem_1523247628},
following \cite{CassHairerLittererTindel2015} closely.
To end this, we regard \eqref{eq_ScaledRDEDrivenByFBm} as an RDE driven by $\RP{w}$
with the coefficients $V^\epsilon_0$, $V^\epsilon_1,\dots,V^\epsilon_d$.
(Except in the rest of this subsection,
we regard \eqref{eq_ScaledRDEDrivenByFBm} as
an RDE driven by $\epsilon \RP{w}$ with the coefficients $V_0,V_1,\dots,V_d$.)
Keeping this in mind,
we introduce a quantity which plays an important role in a Norris-type lemma.
Note that the quantity is defined in the framework of the controlled path theory.
Let $(1+\intPart{1/\Hurst})^{-1}<\alpha<\Hurst$.
Fix $a\in\RealNum^n$ and $0<\theta<1$.
For every $0<\epsilon<1$, define
\begin{align*}
	\mathcal{L}^\epsilon_w(a,\theta,1)
	=
		1
		+L_\theta(w)^{-1}
		+|a|
		+\|(y^\epsilon,J^\epsilon,K^\epsilon)\|_{Q_{\RP{w}}^\alpha}
		+\mathcal{N}_{\RP{w},\alpha}.
\end{align*}
Here,
$\mathcal{N}_{\RP{w},\alpha}
	=
		\sum_{i=1}^{\intPart{1/\Hurst}}
			\|\RP{w}^i\|_{\Hol{i\alpha}}$
and
$Q_{\RP{w}}^\alpha$ stands for the Banach space of
controlled paths with respect to $\RP{w}$.
We refer to \cite[Definition~3]{HairerPillai2013} and \cite[Definition~5.2]{CassHairerLittererTindel2015}
for $L_\theta(w)$, which is called the modulus of $\theta$-H{\"o}lder roughness of $w$,
and to \cite[Definition~5.1]{CassHairerLittererTindel2015}
for $\|(y^\epsilon,J^\epsilon,K^\epsilon)\|_{Q_{\RP{w}}^\alpha}$.

Although we do not discuss detail of $\mathcal{L}^\epsilon_w(a,\theta,1)$ for concise,
we note that expectations of $r$th power of $\mathcal{L}^\epsilon_w(a,\theta,1)$
are bounded in $\epsilon$ (\lref{prop_1529556727})
and it gives good estimate of $\langle v, C^\epsilon v \rangle_{\RealNum^n}$ (\lref{prop_1519378363}).
Combining these two lemmas, we can prove \lref{lem_1523247628}.

Let us start to prove \lref{lem_1523247628} with the next lemma.
\begin{lemma}\label{prop_1529556727}
	For every $1<r<\infty$, we have
	\begin{align*}
		\sup_{0<\epsilon<1}
			\expect[\mathcal{L}_w^\epsilon(a,\theta,1)^r]
		<
			\infty.
	\end{align*}
\end{lemma}

\begin{proof}
	We see $\expect[L_\theta(w)^{-r}]<\infty$ for all $r>1$
	from \cite[Lemma~3]{HairerPillai2013} and \cite[Corollary~5.10]{CassHairerLittererTindel2015}.
	We obtain
	$
		\sup_{0<\epsilon<1}
			\expect[\|(y^\epsilon,J^\epsilon,K^\epsilon)\|_{Q_{\RP{w}}^\alpha}^r]
		<
			\infty
	$
	for all $r>1$
	by reading carefully \cite[Corollary~8.1]{CassHairerLittererTindel2015}
	and using \lref{lem_1523340339}.
	Finally, since $\|\RP{w}^i\|_{\Hol{i\alpha}}^{1/i}$ has a Gaussian tail,
	we see $\expect[\mathcal{N}_{\RP{w},\alpha}^r]<\infty$.
	The proof finished.
\end{proof}

Before stating next lemma, we make a remark.
\begin{remark}
	For every $v=(v_1,\dots,v_n)^\top\in\RealNum^n$, we have
	\begin{align}\label{eq_1519365059}
		\langle v, C^\epsilon v \rangle_{\RealNum^n}
		=
			\sum_{i=1}^d
			\int_0^1
				\langle
					v,K^\epsilon_s V^\epsilon_i(y^\epsilon_s)
				\rangle_{\RealNum^n}^2\,
				ds,
	\end{align}
	which follows from
	\begin{align*}
		\langle v, C^\epsilon v \rangle_{\RealNum^n}
		&=
			\int_0^1
				\langle
					v,
					K^\epsilon_s \sigma^\epsilon(y^\epsilon_s)
					\{K^\epsilon_s\sigma^\epsilon(y^\epsilon_s)\}^\top v
				\rangle_{\RealNum^n}\,
				ds
	\end{align*}
	and
	\begin{align*}
		\langle
			v,
			K^\epsilon_s \sigma^\epsilon(y^\epsilon_s)
			\{K^\epsilon_s\sigma^\epsilon(y^\epsilon_s)\}^\top v
		\rangle_{\RealNum^n}
		&=
			\langle
				\{K^\epsilon_s\sigma^\epsilon(y^\epsilon_s)\}^\top v,
				\{K^\epsilon_s\sigma^\epsilon(y^\epsilon_s)\}^\top v
			\rangle_{\RealNum^d}\\
		&=
			\sum_{i=1}^d
				\langle
					v,K^\epsilon_s V^\epsilon_i(y^\epsilon_s)
				\rangle_{\RealNum^n}^2.
	\end{align*}
\end{remark}

We denote by $\vecFields^\epsilon_m$ and $\vecFields^\epsilon$ sets of vectors
which are defined by replacing $V_i$ by $V^\epsilon_i$ in \dref{def_1541388669}.
Then, we see relationship of $\mathcal{L}_w^\epsilon(a,\theta,1)$
and $\langle v, C^\epsilon v \rangle_{\RealNum^n}$
from the following lemma.
\begin{lemma}\label{prop_1519378363}
	Let $m\in\NaturalNum$.
	For every $0<\epsilon<1$, $W\in\vecFields^\epsilon_m$,
	$v\in\RealNum^n$ with $|v|=1$ and $0\leq s\leq 1$, we have
	\begin{align*}
		|
			\langle
				v,K^\epsilon_s W(y^\epsilon_s)
			\rangle_{\RealNum^n}
		|
		\leq
			c_m
			\mathcal{L}_w^\epsilon(a,\theta,1)^{\mu(m)}
			\langle v, C^\epsilon v \rangle_{\RealNum^n}^{\pi(m)},
	\end{align*}
	where $c_m$, $\mu(m)$ and $\pi(m)$ are certain positive constants
	independent of $\epsilon$, $W$, $v$ and $s$.
\end{lemma}

\begin{proof}
	The proof is done by induction on $m$.
	Let $m=0$. Then, $W=V^\epsilon_i$ for some $1\leq i\leq d$.
	Since
	$
		f^\epsilon_i
		=
			\langle
				v,K^\epsilon V^\epsilon_i(y^\epsilon)
			\rangle_{\RealNum^n}
	$
	is $\alpha$-H{\"o}lder continuous,
	we can use \cite[Lemma~A.3]{HairerPillai2011} to obtain
	\begin{align*}
		\|f^\epsilon_i\|_{\infty}
		\leq
			2
			\|f^\epsilon_i\|_{\Hol{\alpha}}^{1/(2\alpha+1)}
			\|f^\epsilon_i\|_{L^2([0,1];\RealNum)}^{2\alpha/(2\alpha+1)}.
	\end{align*}
	Since
	\begin{align*}
		\|f^\epsilon_i\|_{\Hol{\alpha}}
		\leq
			c
			\{1+\|K^\epsilon\|_{\Hol{\alpha}}\}
			\{|a|+\|y^\epsilon\|_{\Hol{\alpha}}\}
		\leq
			c\mathcal{L}^\epsilon_w(a,\theta,1)^2
	\end{align*}
	holds
	and
	$
		\|f^\epsilon_i\|_{L^2([0,1];\RealNum)}
		\leq
			\langle v, C^\epsilon v \rangle_{\RealNum^n}^{1/2}
	$
	follows from \eqref{eq_1519365059},
	we have
	\begin{align*}
		\|f^\epsilon_i\|_{\infty}
		\leq
			2c^{1/(2\alpha+1)}
			\mathcal{L}^\epsilon_w(a,\theta,1)^{2/(2\alpha+1)}
			\langle v, C^\epsilon v \rangle_{\RealNum^n}^{\alpha/(2\alpha+1)}.
	\end{align*}
	This is the conclusion for $m=0$.

	Assuming the conclusion to hold for $m-1$, we will prove it for $m$.
	Note that for every $W\in\vecFields^\epsilon_m$,
	there exists $U\in \vecFields^\epsilon_{m-1}$ and $0\leq i\leq d$ such that
	$W=[V^\epsilon_i,U]$.
	We have
	\begin{multline*}
		\langle v,K^\epsilon_t U(y^\epsilon_t)\rangle_{\RealNum^n}
		-
		\langle v,U(a)\rangle_{\RealNum^n}\\
		=
			\sum_{i=1}^d
				\int_0^t
					\langle v,K^\epsilon_s[V^\epsilon_i,U](y^\epsilon_s)\rangle_{\RealNum^n}\,
					dw^i_s
			+
			\int_0^t
				\langle v,K^\epsilon_s[V_0,U](y^\epsilon_s)\rangle_{\RealNum^n}\,
				ds.
	\end{multline*}
	From a Norris-type lemma (\cite[Theorem~5.6]{CassHairerLittererTindel2015}, \cite[Theorem~3.1]{HairerPillai2013}),
	there exist positive constant $Q$ and $R$ such that
	\begin{align*}
		\|\langle v,K^\epsilon_\bullet W(y^\epsilon_\bullet)\rangle_{\RealNum^n}\|_{\infty}
		&\leq
			M
			\mathcal{L}_w^\epsilon(a,\theta,1)^Q
			\|
				\langle v,K^\epsilon_\bullet U(y^\epsilon_\bullet)\rangle_{\RealNum^n}
				-
				\langle v,U(a)\rangle_{\RealNum^n}
			\|_{\infty}^R\\
		&\leq
			M
			\mathcal{L}_w^\epsilon(a,\theta,1)^Q
			(
				2
				c_{m-1}
				\mathcal{L}_w^\epsilon(a,\theta,1)^{\mu(m-1)}
				\langle v, C^\epsilon v \rangle_{\RealNum^n}^{\pi(m-1)}
			)^R\\
		&=
			M
			(2c_{m-1})^R
			\mathcal{L}_w^\epsilon(a,\theta,1)^{Q+\mu(m-1)R}
			\langle v, C^\epsilon v \rangle_{\RealNum^n}^{\pi(m-1)R}
	\end{align*}
	for some $M$ depending only on $d$ and $n$.
	This is the conclusion for $m$.
	The proof finished.
\end{proof}

We are now in a position to show \lref{lem_1523247628}.
\begin{proof}[Proof of \lref{lem_1523247628}]
	Let $\nu$ be a positive constant specified later.
	We will show that, for every $1<r<\infty$, there exist positive constants $c_{r,1}$ and $c_{r,2}$ such that
	\begin{gather}
		\label{eq_1528102062}
		\prob(|C^\epsilon|>1/\xi)
		\leq
			c_{r,1}
			\xi^r,\\
		\label{eq_1528102018}
		\sup_{|v|=1}
			\prob
				(
					\langle v,C^\epsilon v\rangle_{\RealNum^n}
					<
						\xi
				)
		\leq
			c_{r,2}
			(\epsilon^{-\nu})^r
			\xi^r
	\end{gather}
	for any $0<\epsilon<1$ and $0<\xi<1$.
	Due to \lref{lem_1528115503},
		 these two estimates are sufficient for \lref{lem_1523247628}.
	Indeed the assertion holds with $\mu=\nu+1$.

	Because \eqref{eq_1528102062} follows from \lref{lem_1523340339},
	we show \eqref{eq_1528102018} in the rest of proof.
	Since the vector fields $V_0,V_1,\dots,V_d$ satisfy the H{\"o}rmander condition at $a$,
	we can choose $W_1,\dots,W_n\in\vecFields$ so that $W_1(a),\dots,W_n(a)$ linearly spans $\RealNum^n$.
	Then, for every $1\leq k\leq n$,
	there exists a non-negative integer $i_k\in\NaturalNum$ such that $W_k\in\vecFields_{i_k}$.
	From definition of $\vecFields^\epsilon_{i_k}$, we can choose a positive constant $\rho_k$
	such that $W^\epsilon_k=\epsilon^{\rho_k}W_k\in\vecFields^\epsilon_{i_k}$.
	Set
	$\rho
		=
			\max\{\rho_1,\dots,\rho_n\}
	$.
	We define
	\begin{align*}
		\phi(u)
		&=
			\max_{1\leq k\leq n}
				|\langle u,W_k(a)\rangle_{\RealNum^n}|,
		&
		\phi^\epsilon(u)
		=
			\max_{1\leq k\leq n}
				|\langle u,W^\epsilon_k(a)\rangle_{\RealNum^n}|
	\end{align*}
	for every $u\in\RealNum^n$ with $|u|=1$.
	Then,
	$
		\phi^\epsilon(u)
		\geq
			\epsilon^\rho
			\phi(u)
	$
	for any $|u|=1$.
	Since $W^\epsilon_1(a),\dots,W^\epsilon_n(a)$ spans linearly $\RealNum^n$
	and $\phi^\epsilon$ is continuous,
	$\phi^\epsilon$ attains the positive minimum.
	Let $c_m$, $\mu(m)$ and $\pi(m)$ be the same symbols as \lref{prop_1519378363}
	and set
	$c=\max\{c_{i_1},\dots,c_{i_n}\}$,
	$\mu=\max\{\mu_{i_1},\dots,\mu_{i_n}\}$,
	and $\pi=\min\{\pi_{i_1},\dots,\pi_{i_n}\}$.
	Set $\nu=\rho/\pi$.

	We choose $v\in\RealNum^n$ with $|v|=1$ and $\epsilon>0$ arbitrarily.
	For $v$ and $\epsilon$, there exists $1\leq k_0\equiv k_0(v,\epsilon)\leq n$ such that
	$\phi^\epsilon(v)=|\langle v,W^\epsilon_{k_0}(a)\rangle_{\RealNum^n}|$.
	From \lref{prop_1519378363} and the above, we have
	\begin{align*}
		\prob
			(
				\langle v,C^\epsilon v\rangle_{\RealNum^n}
				<
					\xi
			)
		\leq
			\prob
				(
					|\langle v,W^\epsilon_{k_0}(a)\rangle_{\RealNum^n}|
					<
						c_{i_{k_0}}
						\mathcal{L}_w^\epsilon(a,\theta,1)^{\mu(i_{k_0})}
						\xi^{\pi(i_{k_0})}
				)
	\end{align*}
	and
	\begin{align*}
		\{
			|\langle v,W^\epsilon_{k_0}(a)\rangle_{\RealNum^n}|
			<
				c_{i_{k_0}}
				\mathcal{L}_w^\epsilon(a,\theta,1)^{\mu(i_{k_0})}
				\xi^{\pi(i_{k_0})}
		\}
		&\subset
			\{
				\epsilon^\rho
				\phi(v)
				<
					c
					\mathcal{L}_w^\epsilon(a,\theta,1)^\mu
					\xi^\pi
			\}\\
		&=
			\{
				\phi(v)
				<
					c
					\mathcal{L}_w^\epsilon(a,\theta,1)^\mu
					\epsilon^{-\rho}
					\xi^\pi
			\}\\
		&\subset
			\{
				\min_{|v|=1}\phi(v)
				<
					c
					\mathcal{L}_w^\epsilon(a,\theta,1)^\mu
					\epsilon^{-\rho}
					\xi^\pi
			\}.
	\end{align*}
	From the Markov inequality, we see
	\begin{align*}
		\prob
			(
				\langle v,C^\epsilon v\rangle_{\RealNum^n}
				<
					\xi
			)
		\leq
			\frac{1}{\{\min_{|v|=1}\phi(v)\}^{r/\pi}}
			c^{p/\pi}
			\expect[\mathcal{L}_w^\epsilon(a,\theta,1)^{\mu r/\pi}]
			(\epsilon^{-\nu})^r
			\xi^r.
	\end{align*}
	Noting $\expect[\mathcal{L}_w^\epsilon(a,\theta,1)^{\mu r/\pi}]$
	are bounded from above in $\epsilon$, we obtain \eqref{eq_1528102018}.
\end{proof}

\subsection{Covariance matrix of solution of scaled-shifted RDE driven by fBm}\label{sec_1542002443}
In this subsection, we will show \pref{prop_1526956280}.
Let $1/\Hurst<p<\intPart{1/\Hurst}+1$ and $(\Hurst+1/2)^{-1}<q<2$ satisfy $1/p+1/q>1$.
We start our discussion with making a remark on continuity of $Q$
on $\geoRPs{p}(\RealNum^d)\times\RealNum\langle\lambda\rangle$,
where $\lambda$ is a one-dimensional path defined by $\lambda_t=t$.
Recall that the solution maps $(\RP{x},c\lambda)\mapsto y, J, K$ are continuous on $\geoRPs{p}(\RealNum^d)\times\RealNum\langle\lambda\rangle$,
where $y$, $J$ and $K$ are solutions to \eqref{eq_1526972098}, \eqref{eq_1539835811} and \eqref{eq_1539835827} driven by $(x,c\lambda)$, respectively.
Furthermore, $Dy_1$ is continuous in $y$, $J$ and $K$
since the right-hand side of \eqref{eq_1526971767} is Young integration
and Young integration is continuous in both integrands and integrators.
Note that the embedding $\CM\subset C_0^{\var{q}}([0,1];\RealNum^d)$ with $1/p+1/q>1$ holds.
Hence, we see that $Q$ is continuous on $\geoRPs{p}(\RealNum^d)\times\RealNum\langle\lambda\rangle$.
We are now in a position to prove \pref{prop_1526956280}.
\begin{proof}[Proof of \pref{prop_1526956280}]
	From the continuity of $Q$ at $(\RP{\gamma},0)\in\geoRPs{p}(\RealNum^d)\times\RealNum\langle\lambda\rangle$
	and the assumption on $\gamma$,
	there exists an open set $O\subset \geoRPs{p}(\RealNum^d)\times\RealNum\langle\lambda\rangle$ such that
	\begin{align}\label{eq_1527063264}
		Q(\tau_\gamma(\RP{x}),\RP{k})\geq c_1 \idMat
	\end{align}
	for any $(\RP{x},k)\in O$.
	Note that $O$ contains $(\RP{0}, 0)$.
	(Here, $(\tau_\gamma(\RP{x}),\RP{k})$ is the Young pairing
	and $(\RP{x},k)$ is a pair.
	In this proof, both of them will appear.)

	We decompose $\expect[|\det \tilde{Q}^\epsilon|^{-r}]$ into expectations on
	$U_\epsilon=\{w\in\probSp\mid(\epsilon\RP{w},\epsilon^{1/\Hurst}\lambda)\in O\}$
	and $U_\epsilon^\complement$.
	First, we study the expectation on $U_\epsilon$.
	Since
	$
		\tilde{Q}^\epsilon
		=
			\epsilon^2 Q(\tau_\gamma(\epsilon\RP{w}),\epsilon^{1/\Hurst}\RP{\lambda})
	$
	and \eqref{eq_1527063264},
	we see
	\begin{align*}
		\det \tilde{Q}^\epsilon
		=
			\epsilon^{2n}
			\det Q(\tau_\gamma(\epsilon\RP{w}),\epsilon^{1/\Hurst}\RP{\lambda})
		\geq
			\epsilon^{2n}
			c_1^n,
	\end{align*}
	which implies
	\begin{align}\label{eq_1541494253}
		\expect[|\det \tilde{Q}^\epsilon|^{-r};U_\epsilon]
		\leq
			(c_1\epsilon^2)^{-nr}.
	\end{align}

	Next we consider the expectation on $U_\epsilon^\complement$.
	From the H\"{o}lder inequality, we have
	\begin{align*}
		\expect[|\det \tilde{Q}^\epsilon|^{-r};U_\epsilon^\complement]
		\leq
			\expect[|\det \tilde{Q}^\epsilon|^{-2r}]^{\frac{1}{2}}
			\prob(U_\epsilon^\complement)^{\frac{1}{2}}.
	\end{align*}
	Since the rate function of
	the Schilder-type large deviation principle is good,
	we see that there exists a positive constant $c_2$ such that
	\begin{align*}
		\prob(U_\epsilon^\complement)
		=
			\prob((\epsilon\RP{w},\epsilon^{1/\Hurst}\lambda)\in O^\complement)
		=
			\hat\nu_{\epsilon}(O^\complement)
		\leq
			\exp
				\left(
					-
					\frac{c_2}{2\epsilon^2}
				\right)
	\end{align*}
	for small $\epsilon>0$.
	Here, we choose $1<q<\infty$ so that
	$
		(q-1)\|\gamma\|_\CM^2
		<
			c_2
	$
	and set $1<q'<\infty$ as the H\"{o}lder conjugate of $q$.
	The Girsanov theorem and \pref{prop_1526540975} imply
	\begin{align*}
		\expect[|\det \tilde{Q}^\epsilon|^{-2r}]^{\frac{1}{2}}
		&=
			\expect
				\left[
					|\det Q^\epsilon|^{-2r}
					\exp
						\left(
							\left\langle w,\frac{\gamma}{\epsilon} \right\rangle
							-
							\frac{\|\gamma\|_\CM^2}{2\epsilon^2}
						\right)
				\right]^{\frac{1}{2}}\\
		&\leq
			\expect[|\det Q^\epsilon|^{-2rq'}]^{\frac{1}{2q'}}
			\expect
				\left[
					\exp
						\left(
							q
							\left\{
								\left\langle w,\frac{\gamma}{\epsilon} \right\rangle
								-
								\frac{\|\gamma\|_\CM^2}{2\epsilon^2}
							\right\}
						\right)
				\right]^{\frac{1}{2q}}\\
		&\leq
			(c_3\epsilon^{-\mu})^r
			\exp
				\left(
					\frac{(q-1)\|\gamma\|_\CM^2}{4\epsilon^2}
				\right).
	\end{align*}
	From the above, we see
	\begin{align}\label{eq_1541494293}
		\expect[|\det \tilde{Q}^\epsilon|^{-r};U_\epsilon^\complement]
		&\leq
			c_3^r
			\epsilon^{-r\mu}
			\exp
				\left(
					\frac{(q-1)\|\gamma\|_\CM^2}{4\epsilon^2}
				\right)
			\exp
				\left(
					-
					\frac{c_2}{4\epsilon^2}
				\right).
	\end{align}

	The estimates \eqref{eq_1541494253} and \eqref{eq_1541494293} imply
	\begin{align*}
		\expect[|\det \epsilon^{-2} \tilde{Q}^\epsilon|^{-r}]
		&=
			\epsilon^{2nr}
			\expect[|\det \tilde{Q}^\epsilon|^{-r}]\\
		&\leq
			\epsilon^{2nr}
			\left\{
				(c_1\epsilon^2)^{-nr}
				+
				c_3^r
				\epsilon^{-r\mu}
				\exp
					\left(
						-
						\frac{c_2-(q-1)\|\gamma\|_\CM^2}{4\epsilon^2}
					\right)
			\right\}\\
		&=
			c_1^{-nr}
			+
			c_3^r
			\epsilon^{(2n-\mu)r}
			\exp
				\left(
					-
					\frac{c_2-(q-1)\|\gamma\|_\CM^2}{4\epsilon^2}
				\right).
	\end{align*}
	The right-hand side is bounded as $\epsilon\searrow 0$.
	The proof has finished.
\end{proof}


\section{Off-diagonal short time asymptotics}
\label{sec.pf}

In this section, following Watanabe \cite{Watanabe1987},
we prove the short time asymptotics of kernel function
$p_t(a,a')$ when $a \neq a'$ and $1/4<\Hurst\leq 1/2$.
Unlike in \cite{Watanabe1987},
we can localize around the energy minimizing path
in the geometric rough path space in this paper,
since Lyons-It\^o map is continuous in this setting.
(The case $\Hurst>1/2$ was done in \cite{Inahama2016b}
and the case $1/3<\Hurst\leq 1/2$ was done in \cite{Inahama2016c}.)

Hereafter in this section, we fix $1/4<\Hurst\leq 1/2$.
Let $(1+\intPart{1/\Hurst})^{-1}<\alpha<\Hurst$
and choose $m\in\NaturalNum_+$ such that $\Hurst-\alpha>2/m$.
Set $p=1/\alpha$ and $q=(\Hurst+1/2-1/m)^{-1}$.

\subsection{Localization around energy minimizing path}
\label{subsec.ldp}
Let us introduce a cut-off function for the localization.
Let $\bar{\gamma}\in\CM$ be as in \aref{ass_minimizer} and $\epsilon>0$.
Since $\|\RP{w}^i\|_{\Bes{(i\alpha, 12m/i)}}^{12m/i}$ is an element of
an inhomogeneous Wiener chaos of order $12m$,
so is its Cameron-Martin shift
$\|\tau_{-\bar{\gamma}} (\epsilon \RP{w} )^i \|_{\Bes{(i\alpha, 12m/i)}}^{12m/i}$.
Here, $\tau_{-\bar{\gamma}}$ is the Young translation by $-\bar{\gamma}$.
It is a continuous map from $\geoRPsBes{\alpha,12m}(\RealNum^d)$ to itself.
So, this Wiener functional is defined for almost all $w \in \probSp$.
For any $r \in (1, \infty)$, $L^r$-norm of
this Wiener functional is bounded in $\epsilon$.
Hence, so is its $\SobSp{r}{k}$-norm for any $r, k$.
Due to this fact, the localization is allowed even in the framework of Watanabe distribution theory.
This is the reason why we use this Besov-type norm on the geometric rough path space.
Let $\psi:\RealNum \to [0,1]$ be a smooth function such that
$\psi(u)=1$ if $|u|\leq 1/2$ and $\psi(u)=0$ if $|u|\geq 1$.
For each $\eta >0$ and $\epsilon >0$, we set
\begin{align*}
	\chi_\eta (\epsilon, w)
	=
		\prod_{i=1}^{\intPart{1/\Hurst}}
			\psi
				\left(
					\frac{\|\tau_{-\bar{\gamma}}(\epsilon \RP{w})^i \|_{\Bes{(i\alpha, 12m/i)}}^{12m/i}}{\eta^{12m}}
				\right).
\end{align*}

%
%
The following lemma states that only rough paths sufficiently close to
the lift of the minimizer $\bar{\gamma}$ contribute to the asymptotics.
The keys of the proof are
\textrm{(i)}
the Schilder type large deviations for
fractional Brownian rough path and
\textrm{(ii)}
the Kusuoka-Stroock type estimate of the Malliavin covariance
of $y^{\epsilon}_1$
(\pref{prop_1526540975}).
\begin{lemma}\label{lm.ldpcut}
	Suppose that \aref{ass_HorCon} and \aref{ass_minimizer} hold.
	Then, for any $\eta >0$, there exists $c=c_\eta>0$ such that
	\begin{align*}
		0
		\leq
			\expect
				[
					(1-\chi_\eta(\epsilon,w))
					\cdot
					\delta_{a'}(y^{\epsilon}_1)
				]
		=
			O
				\left(
					\exp
						\left\{
							-\frac{\|\bar{\gamma}\|_\CM^2+c}{2\epsilon^2}
						\right\}
				\right)
			\qquad
			\text{as $\epsilon \searrow 0$.}
	\end{align*}
\end{lemma}

\begin{proof}
	We show the assertion for $1/4<\Hurst\leq 1/3$.
	We can show it for $1/3<\Hurst\leq 1/2$ more easily.

	Set
	$
		p_\epsilon
		=
			\expect
				[
					(1-\chi_\eta(\epsilon,w))
					\delta_{a'}(y^{\epsilon}_1)
				]
	$.
	We take $\eta' >0$ arbitrarily and fix it for a while.
	It is obvious that
	\begin{align*}
		0
		\leq
			p_\epsilon
		=
			\expect
				\left[
					\{1-\chi_\eta(\epsilon,w)\}
					\psi
						\left(
							\frac{|y^{\epsilon}_1-a'|^2}{\eta'^2}
						\right)
			 		\delta_{a'}(y^{\epsilon}_1)
				\right].
	\end{align*}
	Set $A(\xi_1, \xi_2, \xi_3) = 1 - \psi(\xi_1)\psi(\xi_2)\psi(\xi_3)$
	for $\xi = (\xi_1,\xi_2,\xi_3)\in\RealNum^3$
	and write $(\xi_i;i=1,2,3)=(\xi_1,\xi_2,\xi_3)$.
	Set $g(u) = u \vee 0$ for $u \in \RealNum$.
	Then, in the sense of distributional derivative, $g^{\prime \prime} =\delta_0$.
	Take a bounded continuous function $C: \RealNum^n \to \RealNum$ such that
	$C(u_1, \ldots, u_n) = g(u_1- a'_1) g(u_2- a'_2)\cdots g(u_n- a'_n)$
	if $|u -a'| \leq 2\eta'$.
	Then,
	\begin{align*}
		p_\epsilon
		=
			\expect
				\left[
					A
						\left(
							\frac{\|\tau_{-\bar{\gamma}}(\epsilon\RP{w})^i\|_{\Bes{(i\alpha,12m/i)}}^{12m}}{\eta^{12m}};i=1,2,3
						\right)
					\psi
						\left(
							\frac{|y^{\epsilon}_1-a'|^2}{\eta'^2}
						\right)
					(\partial_1^2\cdots\partial_n^2 C)(y^{\epsilon}_1)
				\right].
	\end{align*}

	Now, we use integration by parts formula for generalized expectations
	as in \cite{Watanabe1987, IkedaWatanabe1989} to see that
	$p_\epsilon$ is equal to a finite sum of the following form;
	\begin{multline*}
		p_\epsilon
		=
			\sum_{j,k}
				\expect
					\left[
						F_{j,k}(\epsilon,w)
						\nabla^j A
							\left(
								\frac{\|\tau_{-\bar{\gamma}}(\epsilon\RP{w})^i \|_{\Bes{(i\alpha,12m/i)}}^{12m}}{\eta^{12m}};i=1,2,3
							\right)
					\right.\\
					\left.
						\vphantom{\left(
							\frac{\|\tau_{-\bar{\gamma}}(\epsilon\RP{w})^i \|_{\Bes{(i\alpha,12m/i)}}^{12m}}{\eta^{12m}};i=1,2,3
						\right)}
						\times
						\psi^{(k)}
							\left(
								\frac{|y^{\epsilon}_1-a'|^2}{\eta'^2}
							\right)
						C(y^{\epsilon}_1 )
					\right].
	\end{multline*}
	Here $j=(j_1,j_2,j_3)$ and $k$ run over finite subsets of $\NaturalNum^3$ and $\NaturalNum$, respectively,
	$\nabla^j A=\partial_1^{j_1}\partial_2^{j_2}\partial_3^{j_3}A$
	and $F_{j,k} (\epsilon, w)$ is a polynomial in components of
	the following (i)--(iv):
	(i) $y^{\epsilon}_1$ and its derivatives,
	(ii) $\|\tau_{-\bar{\gamma}}(\epsilon\RP{w})^i\|_{\Bes{(i\alpha, 12m/i)}}^{12m/i}$ and its derivatives,
	(iii) $Q^\epsilon$, which is Malliavin covariance matrix of $y^{\epsilon}_1$ and its derivatives,
	and (iv) $(Q^\epsilon)^{-1}$.
	Note that the derivatives of $(Q^\epsilon)^{-1}$ do not appear.

	From \pref[prop_1529919750]{prop_1526540975}, there exists $\rho>0$ such that
	$|(Q^\epsilon)^{-1}| = O(\epsilon^{-\rho})$ in $L^r$
	as $\epsilon \searrow 0$ for all $1<r<\infty$.
	(Recall a well-known formula to obtain the inverse matrix $A^{-1}$ with the adjugate matrix of $A$
	divided by $\det A$.)
	Therefore, there exists $\rho >0$ such that
	$|F_{j,k} (\epsilon)| = O(\epsilon^{-\rho})$ in any $L^r$-norm.
	($\rho = \rho(r)>0$ may change from line to line.)
	By H\"older's inequality, we have
	\begin{align}
		\label{eq_ldp_pf3}
		p_\epsilon
		&\leq
			\frac{c}{\epsilon^{\rho}}
			\sum_{j,k}
				\expect
					\Bigg[
						\left|
							\nabla^j A
								\left(
									\frac{\|\tau_{-\bar{\gamma}}(\epsilon\RP{w})^i \|_{\Bes{(i\alpha,12m/i)}}^{12m}}{\eta^{12m}};i=1,2,3
								\right)
						\right|^{r'}\\\nonumber
		&\phantom{\leq}\qquad\qquad\qquad\qquad\qquad\qquad\qquad\qquad\qquad
						\times
						\left|
							\psi^{(k)}
								\left(
									\frac{|y^{\epsilon}_1-a'|^2 }{\eta'^2}
								\right)
						\right|^{r'}
					\Bigg]^{1/r'}\\ \nonumber
		&\leq
			\frac{c}{\epsilon^{\rho}}
			\prob
				\left[
					\bigcup_{i=1}^3
						\left\{
			 				\|\tau_{-\bar{\gamma}}(\epsilon \RP{w})^i \|_{\Bes{(i\alpha,12m/i)}}^{1/i}
							\geq
								\frac{\eta}{2^{1/(12m)}}
			 			\right\}
						\cap
						\{|y^{\epsilon}_1-a'|\leq\eta'\}
				\right]^{1/r'}.
	\end{align}
	Here, $1/r +1/r' =1$ and $c=c(r,r', \eta, \eta')$ is a positive constant,
	which may change from line to line.
	Set
	$
		U_{\eta''}
		=
			\bigcap_{i=1}^3
				\{
					\RP{x}\in\geoRPsBes{\alpha,12m}(\RealNum^d)
					\mid
					\|\RP{x}^i \|_{\Bes{(i\alpha, 12m/i)}}^{1/i}<\eta''
				\}
	$
	for $\eta'' >0$.
	Then this forms a fundamental system of open neighborhoods around $(\RP{x}^1,\RP{x}^2,\RP{x}^3) \equiv (0,0,0)$
	with respect to $(\alpha, 12m)$-Besov topology.
	By \pref{prop_1530162489},
	$
		\tau_{\bar{\gamma}}^{-1}(U_{\eta''})
		=
			\{
				\RP{x}\in\geoRPsBes{\alpha,12m}(\RealNum^d)
				\mid
				\tau_{\bar{\gamma}}(\RP{x})\in U_{\eta''}
			\}
	$
	is an open neighborhood of $\bar{\RP{\gamma}}$
	in $(\alpha,12m)$-geometric rough path space.
	The first set on the most right-hand side of \eqref{eq_ldp_pf3} can be written as
	$\{\epsilon\RP{w}\notin\tau_{\bar{\gamma}}^{-1}(U_{ 2^{-1/(12m)}\eta})\}$.

	First taking $\limsup_{\epsilon \searrow 0}\epsilon^2\log$ and then letting $r'\searrow 1$, we obtain
	\begin{multline}
		\label{eq_1531297248}
		\limsup_{\epsilon \searrow 0}
			\epsilon^2
			\log p_\epsilon\\
		\begin{aligned}
			&\leq
				\limsup_{\epsilon \searrow 0}
					\epsilon^2
					\log
						\prob
							\left[
								w\in\probSp\mid
								\epsilon \RP{w}\notin\tau_{\bar{\gamma}}^{-1}(U_{2^{-1/(12m)}\eta}),
								|y^{\epsilon}_1-a'| \leq \eta'
							\right]\\
			&=
				\limsup_{\epsilon \searrow 0}
					\epsilon^2
					\log
						\hat\nu^{\epsilon}
							\Big[
								\Big\{
								 	(\RP{x},l)\in\geoRPsBes{\alpha,12m}(\RealNum^d)\times\RealNum\langle\lambda\rangle
									\,\Big|\,\RP{x}\in\tau_{\bar{\gamma}}^{-1}(U_{2^{-1/(12m)}\eta})^\complement,\\
			&\phantom{=}\qquad\qquad\qquad\qquad\qquad\qquad\qquad\qquad\qquad\qquad\qquad
									|a+\Phi(\RP{x},\RP{l})^1_{0,1}-a'|\leq\eta'
								\Big\}
							\Big]\\
			&\leq
				-
				\inf
					\left\{
						\frac{ \|\gamma \|_\CM^2 }{2}
						\,\middle|\,
						\gamma\in\CM,
						\RP{\gamma}\in\tau_{\bar{\gamma}}^{-1}(U_{2^{-1/(12m)}\eta})^c,
						|a+\Phi (\RP{\gamma},\RP{0})^1_{0,1}-a'|\leq\eta'
					\right\}.
		\end{aligned}
	\end{multline}
	Here, $\Phi: \geoRPs{p} (\RealNum^{d+1}) \to \geoRPs{p} (\RealNum^n)$
	denotes the Lyons-It\^o map that corresponds to the coefficient $[\sigma, b]$
	and we used the embeddings
	$
		\geoRPsBes{\alpha,12m}(\RealNum^d)\times\RealNum\langle\lambda\rangle
		\hookrightarrow
			\geoRPsBes{\alpha,12m}(\RealNum^{d+1})
		\hookrightarrow
			\geoRPs{p}(\RealNum^{d+1})
	$ implicitly.
	In the last inequality we used large deviation upper estimate for a closed set.
	Notice also that $a+\Phi (\RP{\gamma},0)^1 = \phi^0 (\gamma)$.


	Now let $\eta'$ tend to $0$.
	As $\eta'$ decreases, the most right-hand side of \eqref{eq_1531297248} decreases.
	The proof is finished if the limit is strictly smaller than $- \|\bar{\gamma}\|_\CM^2/2$.
	Assume otherwise.
	Then, there exists $\{\gamma_k\}_{k=1}^{\infty} \subset \CM$ such that
	\begin{align*}
		\RP{\gamma}_k
		&\in
			\tau^{-1}_{\bar{\gamma}}(U_{2^{-1/(12m)}\eta})^c,
		&
		|a+\Phi (\RP{\gamma}_k,0)^1_{0,1}-a'|
		&\leq
			\frac{1}{k},
		&
		\liminf_{k\to\infty}
			\left(-\frac{\|\gamma_k\|_\CM^2}{2}\right)
		&\geq
			-\frac{\|\bar{\gamma}\|_\CM^2}{2}.
	\end{align*}
	In particular, $\{\gamma_k\}$ is bounded in $\CM$.
	Hence, by goodness of the rate function,
	the lifts $\{ \RP{\gamma}_k\}$ is precompact in $\geoRPsBes{\alpha, 12m} (\RealNum^d)$.
	By taking a subsequence if necessary,
	we may assume $\{\gamma_k\}$ converges to some $\RP{z}$ in $(\alpha, 12m)$-Besov topology.
	By the continuity of $\Phi$, we have $a+\Phi (\RP{z},0)^1_{0,1}=a'$.
	Since $\RP{z}\in \tau_{\bar{\gamma}}^{-1}(U_{2^{-1/(12m)}\eta})^c$, $\RP{z}\neq\bar{\RP{\gamma}}$.
	From the lower semicontinuity of the rate function, we see that $\RP{z}$ is the lift of some $z \in \CM$
	and $\|z\|_\CM^2/2\leq\|\bar{\gamma}\|_{\CM}^2/2$.
	This clearly contradicts \aref{ass_minimizer}.
\end{proof}

\subsection{Proof of \tref{thm_MAIN}}

Now, let us calculate the kernel $p_t(a, a')$.
We see that
$
	p_{\epsilon^{1/\Hurst}}(a, a')
	=
		\expect[\delta_{a'} ( y_1^{\epsilon})]
	=
		I_1 (\epsilon)+I_2 (\epsilon)
$,
where
\begin{align*}
	I_1 (\epsilon)
	&=
		\expect
			[
 				\delta_{a'} ( y_1^{\epsilon} )   \chi_{\eta} (\epsilon, w)
			],
	&
	I_2 (\epsilon)
	&=
		\expect
			[
				\delta_{a'} ( y_1^{\epsilon} )  \{ 1- \chi_{\eta} (\epsilon, w) \}
 			].
\end{align*}
As we have shown in \lref{lm.ldpcut},
the second term $I_2 (\epsilon)$ on the right hand side
does not contribute to the asymptotic expansion
for any $\eta >0$.
So, we have only to calculate the first term $I_1 (\epsilon)$
for some $\eta >0$.

However, the proof of the asymptotic expansion
of
$I_1 (\epsilon)$ in the case $\Hurst \in (1/4, 1/3]$ is
essentially the same as in the case $\Hurst \in (1/3, 1/2]$
(see \cite[Subsections 5.2--5.3]{Inahama2016c}).
Therefore, for the sake of brevity,
we will give a sketch of proof only.

\begin{proof}[Sketch of proof of \tref{thm_MAIN}]
By Cameron-Martin formula, we have
\begin{align*}
	I_1  (\epsilon)
	=
		\expect
			\bigg[
				\exp
					\bigg(
						-\frac{ \|  \bar{\gamma} \|^2_{\CM}}{2\epsilon^2}
						-\frac{1}{\epsilon} \langle \bar{\gamma}, w \rangle
					\bigg)
				\delta_{a'} ( \tilde{y}_1^{\epsilon} )
				\chi_{\eta} \bigg(\epsilon, w + \frac{\bar{\gamma}}{\epsilon}\bigg)
	 		\bigg].
\end{align*}
Moreover, there exists $ \bar\nu \in \RealNum^n$ such
that $\langle \bar{\gamma}, w \rangle
= \langle \bar\nu , \phi^1_1 (w,\bar{\gamma}) \rangle$ for all $w$,
where the inner product on the right hand side is a
standard one on $\RealNum^n$.
(In fact, $\bar\nu$ is a covector that appears
in the Lagrange multiplier method for
$\phi^0_1$
and
$\gamma \mapsto \|\gamma\|^2_{\CM} /2$
at
$\bar{\gamma}$.
Note that $\phi^1_1$ is a continuous extension of
$\gamma \mapsto D_\gamma \phi^0_1(\bar{\gamma})$.)
Hence,  we have
\begin{align*}
	I_1(\epsilon)
	&=
		\exp
			\bigg(-\frac{\|\bar{\gamma} \|^2_{{\CM}}}{2\epsilon^2}  \bigg)
		\expect
			\bigg[
				\exp
					\bigg(
						-\frac{1}{\epsilon} \langle \bar\nu , \phi^1_1  \rangle
					\bigg)
				\delta_{a'} ( a' +\epsilon \phi_1^1    +  r_{\epsilon, 1}^{2} )
				\chi_{\eta} \bigg(\epsilon, w + \frac{\bar{\gamma}}{\epsilon}\bigg)
			\bigg]\\
	&=
		\frac{1}{\epsilon^n}
		\exp
			\bigg(-\frac{\|\bar{\gamma} \|^2_{{\CM}}}{2\epsilon^2}  \bigg)
		\expect
			\bigg[
				\exp
					\bigg(
						- \frac{1}{\epsilon} \langle \bar\nu , \phi^1_1  \rangle
					\bigg)
			\delta_{0} (\phi_1^1 + \epsilon^{-1}   r_{\epsilon, 1}^{2} )
			\chi_{\eta} \bigg(\epsilon, w + \frac{\bar{\gamma}}{\epsilon}\bigg)
		\bigg]\\
	&=
		\frac{1}{\epsilon^n}
		\exp
			\bigg(-\frac{\|\bar{\gamma} \|^2_{{\CM}}}{2\epsilon^2}  \bigg)
		\expect
			\bigg[
				\exp
					\bigg(
						\frac{\langle \bar\nu ,  r_{\epsilon, 1}^{2}   \rangle}{\epsilon^2}
					\bigg)
				\delta_{0} (  \phi_1^1 + \epsilon^{-1} r_{\epsilon, 1}^{2} )
				\chi_{\eta} \bigg(\epsilon, w + \frac{\bar{\gamma}}{\epsilon}\bigg)
			\bigg]\\
	&=
		\frac{1}{\epsilon^n}
		\exp
			\bigg(-\frac{\|\bar{\gamma} \|^2_{{\CM}}}{2\epsilon^2}  \bigg)
		\expect
			\bigg[
				F(\epsilon, w)
				\delta_{0} \bigg(  \frac{ \tilde{y}_1^{\epsilon} - a'}{\epsilon } \bigg)
			\bigg],
\end{align*}
where
\begin{align*}
	F(\epsilon, w)
	=
		\exp
			\bigg(
				\frac{\langle \bar\nu ,  r_{\epsilon, 1}^{2}   \rangle}{\epsilon^2}
			\bigg)
		\chi_{\eta} \bigg(\epsilon, w + \frac{\bar{\gamma}}{\epsilon}\bigg)
		\psi \bigg(  \frac{1}{\eta'^2}  \bigg| \frac{\tilde{y}_1^{\epsilon} - a'}{\epsilon}   \bigg|^2 \bigg)
\end{align*}
for any positive constant $\eta'$.
Here we have used $\delta_0(\bullet)=\psi(|\bullet|^2/\eta'^2)\delta_0(\bullet)$.

By a slight modification of
Watanabe's asymptotic expansion theory and
uniform non-degeneracy
(\pref{prop_1526956280}),
$\delta_{0} ( ( \tilde{y}_1^{\epsilon} - a')/ \epsilon  )$
admits the following
asymptotic expansion as $\epsilon \searrow 0$ in
$\tilSobSp{-\infty}{}$-topology as follows:
for some $\Phi_{\nu_j}  \in \tilSobSp{-\infty}{}$ ($j=1,2,\ldots$),
\begin{align}\label{eq.180810}
	\delta_{0} \bigg(  \frac{ \tilde{y}_1^{\epsilon} - a'}{\epsilon } \bigg)
	\sim
	\delta_{0} ( \phi_1^1 )
	+
	\epsilon^{\nu_1} \Phi_{\nu_1}
	+ \epsilon^{\nu_2} \Phi_{\nu_2}+\cdots
\end{align}
(see \cite[Theorem 9.3, p.~387]{IkedaWatanabe1989}).
Recall that
$0= \nu_0 < \nu_1 < \nu_2 <\cdots$ are
all the elements of $\Lambda_3$ in increasing order.
Moreover, since
$\phi_1^1$ is a Gaussian random vector whose
covariance matrix equals  $Q (\bar{\gamma})$,
$\delta_{0} ( \phi_1^1 )$
is actually a non-trivial finite measure
with its total mass
 $\expect [\phi_1^1 ]
= (2\pi)^{-n/2} (\det Q (\bar{\gamma}))^{-1/2} >0$.

Therefore, our problem reduces to
showing $F (\epsilon, \bullet)$ belongs to
$\tilSobSp{\infty}{}$
and admits asymptotic expansion
in $\tilSobSp{\infty}{}$-topology with the index set $\Lambda'_3$.
However, this is highly non-trivial since
$\exp (\langle \bar\nu ,  r^2_{ \epsilon,1}   \rangle/\epsilon^2)$
itself does not have nice integrability.
This is why the two technical factors are involved in the definition of $F (\epsilon, \bullet)$.

Let us observe the two technical factors.
Firstly,
$\chi_{\eta} (\epsilon, w +  \bar{\gamma} /\epsilon)  $ and its derivatives vanish
outside $\{  w~|~ \epsilon \RP{w} \in U_{\eta} \}$
and
$\psi \bigl( \eta'^{-2}  \bigl| (\tilde{y}_1^{\epsilon} - a' )/\epsilon   \bigr|^2  \bigr)$
and its derivatives vanish
outside $\{  |r^1_{\epsilon,1 } /\epsilon | \le \eta' \}$.
Secondly, as $\epsilon \searrow 0$,
\begin{align}\label{eq.180811}
	\chi_{\eta} \bigg(\epsilon, w + \frac{\bar{\gamma}}{\epsilon}\bigg)
	&=
		1+ O(\epsilon^M),
	&
	\psi \bigg(  \frac{1}{\eta'^2}  \bigg| \frac{\tilde{y}_1^{\epsilon} - a'}{\epsilon}   \bigg|^2 \bigg)
	&= 1+ O(\epsilon^M)
\end{align}
in
$\SobSp{\infty}{}$-topology
for any (large) $M >0$.
Therefore, these two Wiener functionals
have no influence in the coefficient of
the expansion of $F (\epsilon, \bullet)$.

The expansion of
$
\langle \bar\nu ,  r^2_{ \epsilon,1}   \rangle/\epsilon^2
=  \langle \bar\nu ,  \phi^2_1 + r^{\kappa_3}_{\epsilon,1}/\epsilon^2\rangle$
is indexed by $\Lambda'_2$.
Hence, if the integrability issue were left aside,
we would easily have an expansion of the following type:
\begin{equation}\label{eq.180812}
	\exp
		\bigg(
			\frac{\langle \bar\nu ,  r_{\epsilon, 1}^{2}   \rangle}{\epsilon^2}
		\bigg)
	\sim
		e^{\langle \bar\nu ,  \phi^2_1\rangle}
		(
			1
			+\epsilon^{\rho_1} \Xi_{\rho_1}
			+\epsilon^{\rho_2} \Xi_{\rho_2}+\cdots
		),
\end{equation}
where
$0= \rho_0 < \rho_1 < \rho_2 <\cdots$ are
all the elements of $\Lambda'_3$ in increasing order.
Due to \eqref{eq.180811}, it also should  hold that
\begin{equation}\label{eq.180813}
	F (\epsilon, \bullet)
	\sim
		e^{\langle \bar\nu ,  \phi^2_1\rangle}
		(
			1
			+\epsilon^{\rho_1} \Xi_{\rho_1}
			+\epsilon^{\rho_2} \Xi_{\rho_2}+\cdots
		).
\end{equation}
Once \eqref{eq.180813} is actually shown,
then we prove our main theorem (\tref{thm_MAIN}).
Moreover,  $\alpha_0 := \expect [e^{\langle \bar\nu ,  \phi^2_1\rangle}
\delta_{0} ( \phi_1^1 )] >0$ is finite
 due to Assumption \aref{ass_hessian}.

Roughly speaking, the proof of \eqref{eq.180813} goes
in the following way.
First, Assumption \aref{ass_hessian}
implies $\expect [e^{\langle \bar\nu ,  \phi^2_1\rangle}
\delta_{0} ( \phi_1^1 )] <\infty$
since $\langle \bar\nu ,  \phi^2_1\rangle$
belongs to the inhomogeneous Wiener chaos of order 2
whose main term corresponds to
the Hessian in \aref{ass_hessian}.
Though it is not at all obvious, the Hessian is
actually Hilbert-Schmidt on $\CM \times \CM$.

However, what we want to estimate is something like
\begin{align*}
	\exp
		\bigg(
			\langle \bar\nu ,\phi^2_{1}\rangle
			+\frac{\langle \bar\nu ,  r^{\kappa_3}_{\epsilon,1}   \rangle}{\epsilon^2}
		\bigg)
		\cdot
		\delta_{0}
			\bigg( \phi_1^1+\frac{r^2_{ \epsilon,1}}{\epsilon}  \bigg).
\end{align*}
Here, thanks to the factor
$\chi_{\eta} (\epsilon, w +  \bar{\gamma} /\epsilon)$,
we may assume $\epsilon \RP{w}$ stay in the bounded set
$U_{\eta}$
in the geometric rough path space.
Hence, we can use the deterministic version of
Taylor expansion of Lyons-It\^o map
(\pref{pr.map.rp2.5}) to prove that
$\langle \bar\nu ,  r^{\kappa_3}_{\epsilon,1}   \rangle/\epsilon^2$ and $r^2_{ \epsilon,1}/\epsilon$
are actually so ``small"  that adding them
does not destroy the integrability we need.
\end{proof}

\section{Examples}\label{sec_1545285297}

In this section, we provide two examples
to which our main theorem (\tref{thm_MAIN}) applies.
First, we recall the case of near points
under the ellipticity condition.

\begin{example}
Assume the ellipticity condition
at the starting point $a$, that is,
$\{V_1 (a), \ldots, V_d (a) \}$ linearly spans $\RealNum^n$.
Obviously, this implies  \aref{ass_HorCon}.
If the end point $a'$ is sufficiently close to the starting point $a$,
then \aref{ass_minimizer}, \aref{ass_DeterministicNondeg}
and \aref{ass_hessian} are also satisfied.
Therefore, our main theorem can be used.

This was shown in \cite{Inahama2016b}
when $\Hurst\in (1/2, 1)$.
The same proof  works when $\Hurst\in (1/4, 1/2]$, too.
The key of the proof is the implicit function theorem.
Note that under the ellipticity assumption,
the deterministic Malliavin covariance
$Q (\gamma)$ is never degenerate and
the implicit function theorem is available at every
$\gamma \in \CM$.
\end{example}

Second, we provide an example, in which
the coefficient vector fields satisfy the H\"ormander condition,
but not the ellipticity condition.
This model was already studied in \cite{Driscoll2013}.

\begin{example}\label{exm_Heisen}
(The ``fractional diffusion" process on the Heisenberg group).
Let $d=2$, $n=3$ and
\begin{align*}
	V_0
	&=
		0,
	&
	V_1
	&=
		\frac{\partial}{\partial x^1} + 2x^2 \frac{\partial}{\partial x^3},
	&
	V_2
	&=
		\frac{\partial}{\partial x^2} -2x^1 \frac{\partial}{\partial x^3},
	\qquad
	(x^1, x^2, x^3) \in \RealNum^3.
\end{align*}
Then, \aref{ass_HorCon} is satisfied at every point
since $[V_1,V_2]=-4\partial/\partial x^3$.
Fortunately, we can write down the solutions
for \eqref{eq_RDEdrivenByFBM} and \eqref{eq_1526966836}
for any given initial condition $a =(a^1, a^2, a^3)$ as follows:
\begin{align*}
y_t &= \Bigl( a^1+ \RP{w}^{1, 1}_{0,t}, \,
a^2+  \RP{w}^{1, 2}_{0,t}, \,
a^3 + 2 (a^2\RP{w}^{1, 1}_{0,t} - a^1\RP{w}^{1, 2}_{0,t}  ) +
2(\RP{w}^{2, 21}_{0,t} -\RP{w}^{2, 12}_{0,t})
\Bigr),
\\
\phi^0_t (\gamma)&= \Bigl( a^1+ \gamma^1_t, \,
a^2+ \gamma^2_t,
\,
a^3 + 2 (a^2\gamma^1_t- a^1\gamma^2_t ) +
2(\RP{\gamma}^{2, 21}_{0,t} -\RP{\gamma}^{2, 12}_{0,t}) \Bigr).
\end{align*}
Here, $\RP{\gamma}=
(\RP{\gamma}^1, \ldots, \RP{\gamma}^{\lfloor 1/H \rfloor})$
is the natural lift of
$\gamma=(\gamma^1_t, \gamma^2_t)_{0 \le t \le 1} \in \CM$ by means of
Young integration and
$\RP{\gamma}^{2, ij}$ stands for the $(i,j)$-component of
$\RP{\gamma}^2$.

Let $a = (0,0,0), a'=(\xi, \eta,0) \in \RealNum^3$ with
$(\xi, \eta) \neq (0,0)$.
Then,  \aref{ass_minimizer}, \aref{ass_DeterministicNondeg}
and \aref{ass_hessian} are also satisfied.
(For a proof, see \lref{lem_hyou} below.)

Therefore, by \tref{thm_MAIN} and
\rref{RemDriftZero}, we have the following asymptotics:
\begin{align*}
		p_t((0,0,0), (\xi, \eta,0))
		\sim
			\exp
				\left(
					-
					\frac{ \xi^2 +\eta^2}{2t^{2\Hurst}}
				\right)
			\frac{1}{t^{n \Hurst}}
			\left\{
				\alpha_{0}
				+\alpha_{2} t^{2 \Hurst}
				+\alpha_{4} t^{4 \Hurst}
				+\cdots
			\right\}
	\end{align*}
	as $t \searrow 0$
	for some $\alpha_{0}>0$ and
	 $\alpha_{2j}\in \RealNum$ $(j =1,2,\dots)$.
	 \end{example}

\begin{remark}
	We make two remarks on \exref{exm_Heisen}.
	\begin{itemize}
		\item	The definitions of $V_1$ and $V_2$ slightly
				differ from literature to literature.
		\item	$V_1$ and $V_2$ are not of $C^\infty_b$.
				So, precisely speaking, we cannot use
				\tref{thm_MAIN} directly.
				However, there is no problem since the solution $(y_t)$
				has a very explicit expression.
				(The main purpose of imposing
				$C^\infty_b$-condition is to ensure the existence of
				a global solution.)
	 \end{itemize}
\end{remark}

\begin{lemma}\label{lem_hyou}
	Let the notation be as in \exref{exm_Heisen}.
	Then,  \aref{ass_minimizer}, \aref{ass_DeterministicNondeg}
	and \aref{ass_hessian} are satisfied
	for $a = (0,0,0), a'=(\xi, \eta,0) \in \RealNum^3$ with
	$(\xi, \eta) \neq (0,0)$.
	Moreover, $\bar{\gamma}_t = (\xi, \eta)R(1,t)$ and
	$\|\bar{\gamma}\|_\CM^2 = \xi^2+ \eta^2$.
\end{lemma}

\begin{proof}
First, note that we may assume without loss of
generality that $a'=(\xi, 0,0)$ with $\xi >0$.
The reason is as follows.
If $(w_t)$ is a two-dimensional fBm and
$A \in \mathrm{SO}(2)$, then $\tilde{w} := Aw$ is again
a two-dimensional fBm and
$\RP{w}^{2, 21}_{0,t} -\RP{w}^{2, 12}_{0,t}
= \RP{\tilde{w}}^{2, 21}_{0,t} -\RP{\tilde{w}}^{2, 12}_{0,t}$.
So, if we choose  suitable $A$, the problem
reduces to the case $a'=(\xi, 0,0)$ with $\xi >0$.

Next, we fix some notation.
We write
$G(\gamma) = \phi^0_1 (\gamma)$ for any $\gamma \in \CM$
or any two-dimensional continuous path $\gamma$ such that
the Young ODE \eqref{eq_1526966836} makes sense.
Explicitly,
$
	G(\gamma)
	=
		(
			\gamma^1_1,
			\gamma^2_1,
			2
			\int_0^1
				(\gamma^2_s\,d\gamma^1_s-\gamma^1_s\,d\gamma^2_s)
		)
$.
We denote by $\CM_{\Hurst} (\RealNum^m)$
the Cameron-Martin space of an $m$-dimensional fBm
(so $\CM=\CM_{\Hurst} (\RealNum^2)$).
We  write
$I (\gamma)=\| \gamma \|^2_{\CM_{\Hurst} (\RealNum^m)}/2$
for $\gamma \in \CM_{\Hurst} (\RealNum^m)$, but suppress the dimension $m$.
Recall that there exists a real-valued kernel $K(t,s)=K_H(t,s)$, $0< s < t$,
with the following properties:
\begin{enumerate}
	\item	If $(b_t)$ is an $m$-dimensional Brownian motion,
			then $(w_t)$, defined by $w_t = \int_0^t K(t,s)\,db_s$,
			is an $m$-dimensional fBm with Hurst parameter $\Hurst$.
	\item	Set $(\mathcal{K} f)_t = \int_0^t K(t,s) f_s\, ds$.
			Then, $\mathcal{K}=\mathcal{K}_{\Hurst}$ is a unitary isometry
			from  $L^2 ([0, 1]; \RealNum^m)$ to	$\CM_{\Hurst} (\RealNum^m)$.
	\item	$R(t,t') =  \int_0^{t \wedge t'} K(t,s) K(t',s) \,ds$.
\end{enumerate}
For proofs, see \cite[Section 5.1]{Nualart2006} and
\cite[Section 1.2]{BiaginiHuOksendalZhang2008}.
There, an explicit expression of $K$ is also given,
but it is not used here.

Now we consider a one-dimensional problem
before showing \aref{ass_minimizer}.
Let $m=1$.
We now prove that $\xi R(1, \bullet)$ is a unique
minimizer of $I$
among $\gamma \in \CM_{\Hurst} (\RealNum^1)$
that connect $0$ and $\xi >0$ at time $1$, that is,
\begin{align}\label{eq.argmin}
	\mathop{\mathrm{argmin}}
		\{
			I(\gamma)
			\mid
			\gamma \in \CM_{\Hurst} (\RealNum^1),
			\gamma_1 =\xi
		\}
	=
	\{ \xi R(1, \bullet) \}.
\end{align}
From (2) and (3) above we see that
$(\mathcal{K}K(1,\ast))_\bullet=R(1,\bullet)\in\CM_\Hurst(\RealNum^1)$
and
\begin{align*}
	\| R(1, \bullet) \|_{\CM_{\Hurst} (\RealNum^1)}^2
	=
		\| K(1, \bullet) \|^2_{L^2([0,1];\RealNum^1)}
	=
		R(1,1)
	=
		1.
\end{align*}
Take $\gamma\in\CM_\Hurst(\RealNum^1)$ with $\gamma_1=\xi$ arbitrarily.
Since $\mathcal{K}$ is an isometry, there exists $f\in L^2([0,1];\RealNum^1)$
such that $\gamma=\mathcal{K}(f+\xi K(1,\bullet))$ uniquely.
From $(\mathcal{K}K(1,\ast))_\bullet=R(1,\bullet)$ and $\gamma_1=\xi$,
we have $(\mathcal{K}f)_1=0$,
which implies $K(1,\bullet)$ and $f$
are orthogonal in $L^2([0,1];\RealNum^1)$.
Hence we can deduce
$
	\|\gamma\|_{\CM}^2
	=
		\|f+\xi K(1,\bullet)\|_{L^2([0,1];\RealNum^1)}^2
	=
		\|f\|_{L^2([0,1];\RealNum^1)}^2+\xi^2
$.
The minimum of $I(\gamma)$ is achieved only when $f=0$.
Thus, we have shown \eqref{eq.argmin}.
It is easy to see from \eqref{eq.argmin} that
\aref{ass_minimizer} holds with
$\bar{\gamma} = (\xi, 0) R(1, \bullet) \in \CM$.

Next, we prove \aref{ass_DeterministicNondeg}.
Recall that non-degeneracy of the deterministic
Malliavin covariance matrix $Q(\bar{\gamma})$
of $G$ at $\bar{\gamma}$ is equivalent
to the surjectivity of the tangent map $DG(\bar{\gamma})\colon\CM\to\RealNum^3$.
Note that
$\bar{\gamma} \in \CM_{1/2} (\RealNum^2)$ since $\Hurst>1/4$.
From the third assertion in \cite[Theorem~1.10 in Chapter~1]{Bismut1984a},
we see that $DG (\gamma)$ is surjective for every
$\gamma \in \CM_{1/2} (\RealNum^2)$ with $\gamma \neq 0$
(in the reference, it is assumed that the distribution linearly spanned by $\{V_1, V_2\}$ is fat).
Since $C^1$-paths which start at $0$ are dense in $\CM_{1/2} (\RealNum^2)$,
there are three $C^1$-paths
$h_1, h_2, h_3$ which start at $0$  such that
$\{ D_{h_i}G (\bar{\gamma})\}_{1 \le i \le 3}$ spans $\RealNum^3$.
Recalling that $C^1$-paths which start at $0$ belong to $\CM$
(see \cite{FrizVictoir2006b}),
we conclude that the tangent
map of $G\colon \CM \to \RealNum^3$  at $\bar{\gamma}$
is also surjective, which is equivalent to  \aref{ass_DeterministicNondeg}.

Finally, we prove \aref{ass_hessian}.
Let $\bar{\nu}=\bar{\nu} (\bar{\gamma})\in \RealNum^3$
be the Lagrange multiplier at $\bar{\gamma}$
for $I$ under the condition that $G = a'$.
Let $f\colon(- \epsilon_0, \epsilon_0)\to K_a^{a'}$ be as in \aref{ass_hessian}.
Then, in the same way as \cite[Proposition~5.5]{Inahama2016c}, we have
\begin{align*}
	\left.\frac{d^2}{du^2}\right|_{u=0}
		\frac{\|f(u)\|_\CM^2}{2}
	=
		\|f'(0)\|_\CM^2
		-
		\langle
			\bar{\bar{\nu}},
			D^2G (\bar{\gamma})\langle f'(0),f'(0)\rangle
		\rangle_{\RealNum^3},
\end{align*}
where $D^2G (\bar{\gamma})\langle k,k\rangle=D_k^2 G(\bar{\gamma})$.
It is well-known that
\begin{align*}
	\bar{\nu}^i
	=
		\sum_{j=1}^3
			[G(\bar{\gamma})^{-1}]_{ij}
			D_{\bar{\gamma}}G^j(\bar{\gamma}).
\end{align*}

By straightforward computation, we have
for every $\gamma, h \in \CM$
\begin{align*}
	D_h G (\gamma)
	=
		\left(
			h^1_1,
			h^2_1,
			2
			\int_0^1
				(
					h^2_s\,d\gamma^1_s
					-h^1_s\,d\gamma^2_s
					+\gamma^2_s\,dh^1_s
					-\gamma^1_s\,dh^2_s
				)
		\right).
\end{align*}
Thanks to the explicit forms of $\bar{\gamma}$ and $G$,
we can easily see that $D_{\bar{\gamma}}G^j(\bar{\gamma})=(\xi, 0,0)$.
Since the second component of $\bar{\gamma}$ is zero,
the second and the third components  of
$D_h G (\bar{\gamma})$ does not depend on $h^1$,
the first component of $h$.
This implies that
\begin{align*}
	Q(\bar{\gamma})
	&=
		\begin{pmatrix}
			1 & 0 & 0 \\
			0 & 1 & *\\
			0 & * & *
		\end{pmatrix},
	&
	Q(\bar{\gamma})^{-1}
	&=
		\begin{pmatrix}
			1 & 0 & 0 \\
			0 & * & *\\
			0 & * & *
		\end{pmatrix}
\end{align*}
and hence $\bar{\nu} =(\xi, 0,0)$.
Since $G^1$ is linear, $D^2 G^1 \equiv 0$.
Thus, we have shown that
$\langle \bar{\nu}, D^2G (\bar{\gamma})\langle f'(0),f'(0)\rangle\rangle_{\RealNum^3} =0$,
which clearly implies \aref{ass_hessian}.
\end{proof}

\begin{remark}
	When $\Hurst=1/2$,
	Young ODE \eqref{eq_1526966836} controlled by
	$\gamma \in \CM$ is called skeleton ODE in probability theory
	and has been thoroughly studied
	in both  probability theory and control theory.
	In that case,  \aref{ass_HorCon}--\aref{ass_hessian}
	are known to be a natural set of assumptions with many concrete examples.

	On the other hand, when $\Hurst\neq 1/2$,
	not much seems to be known for this ODE.
	This is the main reason why we cannot
	systematically find examples for our main theorem in this section.
	Therefore, we believe that  a ``fractional version of
	the control theory of ODEs" needs to be investigated
	for future developments of this topic.
	(The fraction version could also be useful in relaxing the
	unique minimizer assumption \aref{ass_minimizer}.)
\end{remark}

\appendix
\section{Technical lemma}\label{sec_1545285314}

The following lemma to be used in the proof of \lref{lem_1523247628}
is a slight modification of \cite[Lemma~2.3.1]{Nualart2006}.
For readers' convenience, we prove it in this section.
\begin{lemma}\label{lem_1528115503}
	Let $\{A^\epsilon\}_{0<\epsilon<1}$ be a family of real symmetric non-negative definite $n\times n$ random matrices.
	Let $\mu_1$ and $\mu_2$ be non-negative constants.
	We assume that for every $1<p<\infty$, there exist positive constants $c_{p,1}$ and $c_{p,2}$ such that
	\begin{align*}
		\prob(|A^\epsilon|>1/\xi)
		&
			\leq
				c_{p,1}
				(\epsilon^{-\mu_1}\xi)^p,
		&
		\sup_{|v|=1}
			\prob(\langle v,A^\epsilon v\rangle_{\RealNum^n}<\xi)
			\leq
				c_{p,2}
				(
					\epsilon^{-\mu_2}
					\xi
				)^p
	\end{align*}
	for any $0<\epsilon<1$ and $0<\xi<1$.
	Then, for all $0<\epsilon<1$, $\mu>\mu_1\vee\mu_2$ and $1<r<\infty$, we have
	\begin{align*}
		\expect[\MinEigenVal(A^\epsilon)^{-r}]^{1/r}
		\leq
			c
			\epsilon^{-\mu}.
	\end{align*}
	Here, $c$ is a positive constant, which is independent of $\epsilon$ and $\mu$ but depends on $r$.
\end{lemma}
\begin{proof}
	Fix $0<\epsilon<1$. Let $\mu>\mu_1\vee\mu_2$ and set
	\begin{align*}
		r_0
		=
			\frac{(1+2n)(\mu_1\vee\mu_2)}{\mu-(\mu_1\vee\mu_2)}
			\vee
			1.
	\end{align*}
	First of all, we show that, for every $r\geq r_0$,
	there exists a positive constant $\tilde{c}_r$ such that
	\begin{align}\label{eq_1529303381}
		\prob(\MinEigenVal(A^\epsilon)<\xi)
		\leq
			\tilde{c}_r
			\epsilon^{-r\mu}
			\xi^{r+1}
	\end{align}
	for all $0<\xi<1$.
	Here, we can choose $\tilde{c}_r$ as a constant
	which is independent of $\epsilon$, $\xi$, and $\mu$ but depends on $r$.
	Note that the restriction $r\geq r_0$ and $\mu>\mu_1\vee \mu_2$
	imply $r\mu\geq (r+1+2n)\mu_i$ for $i=1,2$.
	Note
	\begin{align*}
		\prob(\MinEigenVal(A^\epsilon)<\xi)
		&\leq
			\prob(|A^\epsilon|\geq 1/\xi)
			+
			\prob(\MinEigenVal(A^\epsilon)<\xi,|A^\epsilon|\leq 1/\xi).
	\end{align*}
	The first term satisfies
	$
		\prob(|A^\epsilon|\geq 1/\xi)
		\leq
			c_{r+1,1}(\epsilon^{-\mu_1}\xi)^{r+1}
		\leq
			c_{r+1,1}
			\epsilon^{-r\mu}
			\xi^{r+1}
	$
	from the assumption	and $r\mu\geq (r+1)\mu_1$.
	The second term is estimated as follows.
	We denote by $S^{n-1}$ and $B(v,\rho)$ the unit sphere centered at the origin
	and the open ball with center $v\in\RealNum^n$ and radius $\rho>0$.
	Since $S^{n-1}$ is compact,
	there exists vectors $v_1,\dots,v_m\in S^{n-1}$ so that
	$\{B(v_i,\xi^2/2)\}_{i=1}^m$ covers $S^{n-1}$.
	In addition, $\MinEigenVal(A^\epsilon)=\inf_{v\in S^{n-1}}\langle v,A^\epsilon v\rangle_{\RealNum^n}$ holds.
	From these facts, on the event $\{\MinEigenVal(A^\epsilon)<\xi,|A^\epsilon|<1/\xi\}$,
	we can choose $v\in S^{n-1}$ such that $\langle v,A^\epsilon v\rangle_{\RealNum^n} <\xi$
	with $v\in B(v_i,\xi^2/2)$ for some $1\leq i\leq m$.
	Hence,
	\begin{align*}
			\langle v_i,A^\epsilon v_i\rangle_{\RealNum^n}
			\leq
				\langle v,A^\epsilon v\rangle_{\RealNum^n}
				+
				|\langle v_i-v,A^\epsilon v\rangle_{\RealNum^n}|
				+
				|\langle v_i,A^\epsilon (v_i-v)\rangle_{\RealNum^n}|
			<
				2\xi,
	\end{align*}
	which implies
	\begin{align*}
		\{\MinEigenVal(A^\epsilon)<\xi,|A^\epsilon|<1/\xi\}
		\subset
			\bigcup_{i=1}^m
				\{\langle v_i,A^\epsilon v_i\rangle_{\RealNum^n}<2\xi\}.
	\end{align*}
	Moreover, we see that the number $m$ of the vectors may depend on $\xi$; however $m\leq c_n \xi^{-2n}$ holds,
	where $c_n$ is a constant depending only on $n$.
	Therefore, the term $\prob(\MinEigenVal(A^\epsilon)<\xi,|A^\epsilon|\leq 1/\xi)$ is estimated by
	\begin{align*}
		\sum_{i=1}^m
			\prob(\langle v_i,A^\epsilon v_i\rangle_{\RealNum^n}<2\xi)
		\leq
			m
			c_{r+1+2n,2}
			(\epsilon^{-\mu_2}2\xi)^{r+1+2n}
		\leq
			c_{r,n}
			\epsilon^{-r\mu}
			\xi^{r+1},
	\end{align*}
	where
	$
		c_{r,n}
		=
			2^{r+1+2n}
			c_n
			c_{r+1+2n,2}
	$.
	In the last estimate, we used $r\mu\geq (r+1+2n)\mu_2$.
	The estimates above imply
	\begin{align*}
		\prob(\MinEigenVal(A^\epsilon)<\xi)
		&\leq
			c_{r+1,1}
			\epsilon^{-r\mu}
			\xi^{r+1}
			+
			c_{r,n}
			\epsilon^{-r\mu}
			\xi^{r+1},
	\end{align*}
	which implies \eqref{eq_1529303381} with $\tilde{c}_r=c_{r+1,1}+c_{r,n}$.

	Next, we complete the proof. It follows from \eqref{eq_1529303381} that,
	for all $r\geq r_0$,
	\begin{align*}
		\expect[\MinEigenVal(A^\epsilon)^{-r}]
		&=
			r
			\int_0^\infty
				\xi^{-r-1}
				\prob(\MinEigenVal(A^\epsilon)<\xi)\,
				d\xi\\
		&\leq
			r
			\int_0^1
				\xi^{-r-1}
				\tilde{c}_r
				\epsilon^{-r\mu}
				\xi^{r+1}\,
				d\xi
			+
			r
			\int_1^\infty
				\xi^{-r-1}
				d\xi\\
		&=
			r
			\tilde{c}_r
			\epsilon^{-r\mu}
			+
			1\\
		&\leq
			(1+r\tilde{c}_r)
			\epsilon^{-r\mu}.
	\end{align*}
	For $1\leq r<r_0$, we have
	\begin{align*}
		\expect[\MinEigenVal(A^\epsilon)^{-r}]^{1/r}
		\leq
			\expect[\MinEigenVal(A^\epsilon)^{-r_0}]^{1/r_0}
		\leq
			(1+r_0\tilde{c}_{r_0})^{1/r_0}
			\epsilon^{-\mu}.
	\end{align*}
	These imply the assertion.
\end{proof}

\section*{Acknowledgement}
The first named author was partially supported by JSPS KAKENHI Grant Number JP15K04922.
The second named author was partially supported by JSPS KAKENHI Grant Number JP17K14202.

\end{document}